\documentclass[a4paper,11pt]{article}
\usepackage{amssymb}
\usepackage{amsfonts}
\usepackage{amsmath}

\def\bkR{{\rm I\kern-.17em R}}
\newtheorem{theorem}{Theorem}

\newtheorem{corollary}[theorem]{Corollary}

\newtheorem{definition}{Definition}

\newtheorem{proposition}{Proposition}
\newtheorem{remark}{Remark}

\newenvironment{proof}[1][Proof]{\noindent\textbf{#1.} }{\ \rule{0.5em}{0.5em}}
\begin{document}

\title{A Directional Curvature Formula for Convex Bodies in $\mathbb{R}^{n}$ 
\thanks{%
This paper is dedicated to the memory of dear colleague Vladimir V.
Goncharov (University of \'{E}vora, Portugal), who passed away when this
work was just beginning, we only had an idea of how to prove the main result
in $\mathbb{R}^{3}$.} }
\author{F\'{a}tima F. Pereira \thanks{%
E-mail: fmfp@uevora.pt } \thanks{\emph{Departamento de Matem\'{a}tica,
Escola de Ci\^{e}ncias e Tecnologia, Centro de Investiga\c{c}\~{a}o em Matem%
\'{a}tica e Aplica\c{c}\~{o}es (CIMA), Instituto de Investiga\c{c}\~{a}o e
Forma\c{c}\~{a}o Avan\c{c}ada, Universidade de \'{E}vora, Rua Rom\~{a}o
Ramalho 59, 7000-671, \'{E}vora, Portugal}} \\
%EndAName
\emph{Universidade de \'{E}vora, Rua Rom\~{a}o Ramalho 59, 7000-671, \'{E}%
vora, Portugal}}
\date{}
\maketitle

\begin{abstract}
For a compact convex set $F\subset \mathbb{R}^{n}$, with the origin in its
interior, we present a formula to compute the curvature at a fixed point on
its boundary, in the direction of any tangent vector. This formula is
equivalent to the existing ones, but it is easier to apply.

\bigskip 


\textbf{Mathematical Subject Classification\ (2000)}: 52A20, 53A04
\end{abstract}

\section{Introduction}

In \cite{GP1} the authors proposed some concepts concerning the geometric
structure of a closed convex bounded set $F$, with zero in its interior, in
a Hilbert space $H$. Inspired essentially from the geometry of Banach spaces
(see \cite{Megginson}), they introduced three moduli of local rotundity for
the set $F$, one symmetrical (using the norm of $H$) and two asymmetric
(using the "asymmetric norm" given by the Minkowski functional of $F$).
Using the symmetrical modulus the authors defined the concept of strict
convexity graduated by some parameter $\alpha >0.$ The main numerical
characteristic resulting from these considerations is the curvature (and the
respective curvature radius) of $F$, which shows how rotund the set $F$ is
near a fixed boundary point $\xi $ watching along a given direction $\xi
^{\ast }$. Considering the polar set of $F$, $F^{o}$, they defined also the
modulus of local smoothness and the local smoothness of $F^{o}$. As
well-known (see, for example, \cite{Lovaglia, Megginson, Smulian, Smulian2})
the strict convexity of a convex closed bounded set $F$ with zero in its
interior is strongly related to the smoothness of $F^{o}$, but in \cite{GP1}
that relation was quantified. In particular, a local asymmetric version of
the Lindenstrauss duality theorem \cite[Proposition 4.2]{GP1} was proved
there, which quantitatively establishes the duality between local smoothness
and local rotundity. Thus, the curvature of $F$ can be considered also as a
numerical characteristic of $F^{o}$, showing how sleek $F^{o}$ is in a
neighbourhood of a boundary point $\xi ^{\ast }$ if you look along a
direction $\xi $. Applying this theorem, it was obtained a characterization
of the curvature of $F$ in terms of the second derivative of its dual
Minkowski functional \cite[Proposition 4.4]{GP1}. From what we have just
said, and not only (for more results see \cite{GP1, GP2}), the formula for
curvature is, from a theoretical point of view, very useful, but in practice
it is very difficult to use even in $\mathbb{R}^{2}$ as we can see in \cite[%
Example 8.4]{GP1}. Then, in this paper, we propose, in some sense and for
some kind of convex bodies $F$ (compact convex sets with interior points) in 
$\mathbb{R}^{n}$, $n\geq 2$, an equivalent formula to compute its curvature
but easier to use. Namely, in Theorem \ref{Teo1} below, given $\xi $ at the
boundary of $F$, $\partial F$, near which $\partial F$ is given by an
implicit equation, we present a formula for the curvature of $F$ at $\xi $
in the direction of any tangent vector. For this, and for a fixed tangent
vector, we will consider the intersection curve between $F$ and a suitable
plane, but without using the plane equations or the curve expression. In a
few words, we can say that \cite{GP1} gives us an approximate idea of the
shape of $F$ in a global neighbourhood of $\xi $, while in this article the
exact shape of $F$ near $\xi $ in each tangent direction is obtained.

\bigskip

Before moving on to the work itself, let us review more precisely what is
already done in this area.

A definition for curvature similar to the formula that will be obtained
here, and called directional curvature, appears in \cite{Abatzoglou} for a
(not necessarily convex) \emph{C}$^{2}$-manifold embedded in a Hilbert space.

In \cite[p.14]{Busemann} (see also \cite{Adiprasito, Schneider}), for a
convex body $F$ in $\mathbb{R}^{n}$ $\left( n\geq 2\right) ,$ a smooth point 
$\xi $ in $\partial F$ (smooth means that at $\xi $ there exists only one
supporting hyperplane to $F$), an interior unit normal vector $\xi ^{\ast }$
of $F$ at $\xi ,$ and an unit vector $\xi ^{\ast \ast }$ orthogonal to $\xi
^{\ast },$ H. Busemann considered the $2$-dimensional halfplane%
\begin{equation*}
H\left( \xi ,\xi ^{\ast },\xi ^{\ast \ast }\right) =\left\{ \eta \in \mathbb{%
R}^{n}:\eta =\xi +\lambda \xi ^{\ast }+\mu \xi ^{\ast \ast }\text{ with }%
\lambda ,\mu \in \mathbb{R}\text{ and }\mu \geq 0\right\} ,
\end{equation*}%
which intersects $\partial F$ in a plane convex curve. Denoting by $r_{\eta} $, for $\eta \in H\left( \xi ,\xi ^{\ast },\xi ^{\ast \ast }\right) \cap
\partial F$ near $\xi $, the radius of the circle with centre on the normal
line $\xi +\mathbb{R}^{+}\xi ^{\ast }$ containing both $\xi $ and $\eta $,
the author defined%
\begin{equation*}
\rho _{l}^{\xi ^{\ast \ast }}\left( \xi \right) :=~\underset{\eta
\rightarrow \xi }{\lim \inf \ }r_{\eta }\text{,\ \ \ \ }\rho _{u}^{\xi ^{\ast
\ast }}\left( \xi \right) :=\ \underset{\eta \rightarrow \xi }{\lim \sup \ }r_{\eta }
\end{equation*}%
as the lower and upper curvature radius, respectivelly. If the numbers $%
\gamma _{l}^{\xi ^{\ast \ast }}\left( \xi \right) :=\left( \rho _{l}^{\xi
^{\ast \ast }}\left( \xi \right) \right) ^{-1}$ and $\gamma _{u}^{\xi ^{\ast
\ast }}\left( \xi \right) :=\left( \rho _{u}^{\xi ^{\ast \ast }}\left( \xi
\right) \right) ^{-1}$ (called lower and upper curvature, respectivelly) are
equal and finite, he says that the curvature of $F$ at $\xi $ in direction $%
\xi ^{\ast \ast }$ exists and is equal to the common value.

Differential geometry of intersection curves of two (or more) surfaces in $%
\mathbb{R}^{3}$ (or higher dimension) were studied by many authors (see, for
example, \cite{Alessio, Goldman, Ye} and the bibliography therein). There
are studies in which all the surfaces are defined implicitly, others in
which all are parametrically defined, and others in which there are surfaces
of both types. For this work, we are only interested in those defined
implicitly. At \cite{Alessio, Goldman, Ye} the authors present formulas (or
algorithms) for computing differential geometric properties (such as tangent
vector, normal vector, curvatures and torsion) of the intersection curve. In 
\cite[(5.4)]{Goldman} the author derives a formula for the curvature of the
curve defined by the intersection of $n-1$ implicit surfaces in $\mathbb{R}%
^{n}$. This formula is laborious to apply when the space has dimension $%
n\geq 4,$ because we need to do several operations with the gradients of the
all functions that implicity define the surfaces. On the contrary, the
formula that will be presented in this paper seems easier, as it uses only
the function that defines implicitly the convex body.

\bigskip

In Section 2 of this paper we introduce some notations, definitions and one
example where we can see that the definition of curvature presented in \cite%
{GP1} can give us only an approximate value. In Section 3, we present the
conditions on $F,$ the normal cone to $F$ at a convenient $\xi \in \partial
F $, the tangent hyperplane of $F$ at $\xi $, the definition of directional
curvature, some its properties and its relation to the definition of \cite%
{GP1}. Section 4 is dedicated to the main result of this paper and its
proof. In Section 5 we relate the directional curvature of $F$ with the
radius of a suitable sphere. The relationship between our formula and
Goldman's one is proved in Section 6. Finally, the Section 7 is dedicated to
the examples.

\bigskip

\section{Basic notations and definitions\label{Exemplo}}

We will consider in the space $\mathbb{R}^{n}$, $n\in \mathbb{N}$, $n\geq 2,$
with the usual inner product $\left\langle \cdot ,\cdot \right\rangle $ and
the norm $\left\Vert \cdot \right\Vert ,$ a compact convex set $F$ with the
vector null of $\mathbb{R}^{n}$ (represented by $\mathbf{0}$) in its
interior $\mbox{int}F$. We denote by $F^{o}$ the \emph{polar set} of $F,$
i.e.,

\begin{equation*}
F^{o}:=\left\{ \xi ^{\ast }\in \mathbb{R}^{n}:\left\langle \xi ,\xi ^{\ast
}\right\rangle \leq 1\;\;\forall \xi \in F\right\} \text{.}
\end{equation*}%
Together with the \emph{Minkowski functional} $\rho _{F}\left( \cdot \right) 
$ defined by 
\begin{equation*}
\rho _{F}\left( \xi \right) :=\inf \left\{ \lambda >0:\xi \in \lambda
F\right\}
\end{equation*}%
we introduce the \emph{support function} $\sigma _{F}:\mathbb{R}%
^{n}\rightarrow \mathbb{R}^{+}$, 
\begin{equation*}
\sigma _{F}\left( \xi ^{\ast }\right) :=\sup \left\{ \left\langle \xi ,\xi
^{\ast }\right\rangle :\xi \in F\right\} .
\end{equation*}%
Observe that%
\begin{equation*}
\rho _{F}\left( \xi \right) =\sigma _{F^{o}}\left( \xi \right) \text{,}
\end{equation*}%
and, consequently,%
\begin{equation}
\frac{1}{\left\Vert F\right\Vert }\left\Vert \xi \right\Vert \leq \rho
_{F}\left( \xi \right) \leq \left\Vert F^{o}\right\Vert \left\Vert \xi
\right\Vert \text{, \ \ \ }\xi \in \mathbb{R}^{n}\text{,}
\label{Equivalence}
\end{equation}%
where $\left\Vert F\right\Vert :=\sup \left\{ \left\Vert \xi \right\Vert
:\xi \in F\right\} $. The inequalities (\ref{Equivalence}) mean that $\rho
_{F}\left( \cdot \right) $ is a sublinear functional "equivalent" to the
norm $\left\Vert \cdot \right\Vert $. It is not a norm since $-F\neq F$\ in
general.

As usual, we represent by $\partial F$ the \emph{boundary} of $F.$ In what
follows we will use the so-called \emph{duality mapping} $\mathfrak{J}%
_{F}:\partial F^{o}\rightarrow \partial F$ that associates\ the set \ 
\begin{equation*}
\mathfrak{J}_{F}\left( \xi ^{\ast }\right) :=\left\{ \xi \in \partial
F:\left\langle \xi ,\xi ^{\ast }\right\rangle =1\right\}
\end{equation*}%
with each $\xi ^{\ast }\in \partial F^{o}$. We say that $\left( \xi ,\xi
^{\ast }\right) $ is a \emph{dual pair} when $\xi ^{\ast }\in \partial F^{o}$
and $\xi \in \mathfrak{J}_{F}\left( \xi ^{\ast }\right) $.

The \emph{normal cone} to $F$ at $\xi $, in the sense of Convex Analysis, is
given by 
\begin{equation*}
\mathbf{N}_{F}\left( \xi \right) :=\left\{ \zeta ^{\ast }\in \mathbb{R}%
^{n}:\left\langle \eta -\xi ,\zeta ^{\ast }\right\rangle \leq 0\text{~for
every }\eta \in F\right\},
\end{equation*}%
and the \emph{proximal normal cone} to $F$ at $\xi $ is 
\begin{equation*}
\mathbf{N}_{F}^{P}\left( \xi \right) :=\left\{ \zeta ^{\ast }\in \mathbb{R}%
^{n}:\text{there exists }\sigma \geq 0\text{ such that }\left\langle \eta
-\xi ,\zeta ^{\ast }\right\rangle \leq \sigma \left\Vert \eta -\xi
\right\Vert ^{2}\text{ for every }\eta \in F\right\} .
\end{equation*}
Since $F$ is closed and convex we have (see \cite[Proposition 1.1.10]{CLSW}) 
\begin{equation}
\mathbf{N}_{F}^{P}\left( \xi \right) =\mathbf{N}_{F}\left( \xi \right) .
\label{15}
\end{equation}
It is easy to show that $\mathbf{N}_{F}\left( \xi \right) \cap \partial F^{o} $ is the pre-image of the mapping $\mathfrak{J}_{F}\left( \xi \right) $%
, $\mathfrak{J}_{F}^{-1}\left( \cdot \right) $, calculated at $\xi $.

\bigskip

The \emph{tangent cone} to $F$ at $\xi $ is the polar of $\mathbf{N}_{F}\left( \xi \right)$, since $\mathbf{N}_{F}\left( \xi \right) $ is, in
fact, a cone, it is given by 
\begin{equation*}
\left\{ u\in \mathbb{R}^{n}:\left\langle u,\zeta ^{\ast }\right\rangle \leq 0%
\text{~for every }\zeta ^{\ast }\in \mathbf{N}_{F}\left( \xi \right)
\right\} .
\end{equation*}%
We will only work with the \emph{hyperplane tangent }to the set $F$ at the
point $\xi :$%
\begin{equation}
\mathbf{T}_{F}\left( \xi \right) :=\left\{ u\in \mathbb{R}^{n}:\left\langle
u,\zeta ^{\ast }\right\rangle =0\text{~for every }\zeta ^{\ast }\in \mathbf{N%
}_{F}\left( \xi \right) \right\} .  \label{PlanoT}
\end{equation}

\bigskip

Following \cite[Definition 3.2]{GP1}, for each dual pair $\left( \xi ,\xi
^{\ast }\right) $ the \emph{modulus of rotundity of }$F$ at $\xi $ with
respect to (w.r.t.) $\xi ^{\ast }$ is 
\begin{equation}
\widehat{\mathfrak{C}}_{F}\left( r,\xi ,\xi ^{\ast }\right) :=\inf \left\{
\left\langle \xi -\eta ,\xi ^{\ast }\right\rangle :\eta \in F,\;\left\Vert
\xi -\eta \right\Vert \geq r\right\} \text{, \ \ }r>0\text{,}  \label{Def3.2}
\end{equation}%
and $F$ is said to be \emph{strictly convex }(or \emph{rotund}) at $\xi $\
w.r.t. $\xi ^{\ast }$ if 
\begin{equation}
\widehat{\mathfrak{C}}_{F}\left( r,\xi ,\xi ^{\ast }\right) >0\text{ \ \ for
all \ \ }r>0\text{.}  \label{rotundity}
\end{equation}%
If (\ref{rotundity}) is fulfilled then $\xi $ is an \emph{exposed point} of $%
F$ and the vector $\xi ^{\ast }$\ \emph{exposes} $\xi $ in the sense that
the hyperplane $\left\{ \eta \in \mathbb{R}^{n}:\left\langle \eta ,\xi
^{\ast }\right\rangle =\sigma _{F}\left( \xi ^{\ast }\right) \right\} $
touches $F$ only at the point $\xi $, or, in other words, $\mathfrak{J}%
_{F}\left( \xi ^{\ast }\right) =\left\{ \xi \right\} $. So, in this case, $%
\xi $ is well defined whenever $\xi ^{\ast }$ is fixed.

\begin{definition}[\protect\cite{GP1}]
Fix $\xi ^{\ast }\in \partial F^{o}$, and let $\xi $ be the unique element
of $ \ \mathfrak{J}_{F}\left( \xi ^{\ast }\right) $. The set $F$\ is said to be 
\emph{strictly convex of order 2 }(\emph{at the point}\ $\xi $) \emph{w.r.t.}
$\xi ^{\ast }$ if 
\begin{equation}
\hat{\gamma}_{F}\left( \xi ,\xi ^{\ast }\right) =\liminf_{\substack{ \left(
r,\eta ,\eta ^{\ast }\right) \rightarrow \left( 0+,\xi ,\xi ^{\ast }\right) 
\\ \eta \in \mathfrak{J}_{F}\left( \eta ^{\ast }\right) ,\,\eta ^{\ast }\in
\partial F^{o}}}\frac{\widehat{\mathfrak{C}}_{F}\left( r,\eta ,\eta ^{\ast
}\right) }{r^{2}}>0\text{.}  \label{gama}
\end{equation}
The number
\begin{equation*}
\hat{\varkappa}_{F}\left( \xi ,\xi ^{\ast }\right) =\frac{1}{\left\Vert \xi
^{\ast }\right\Vert }\ \hat{\gamma}_{F}\left( \xi ,\xi ^{\ast }\right)
\end{equation*}%
is said to be the (\emph{square}) \emph{curvature} of $F$\emph{\ at }$\xi
\in \partial F$\emph{\ w.r.t. }$\xi ^{\ast }$.
\end{definition}

\bigskip

\paragraph{An example \label{Mot}}

Consider the compact convex set 
\begin{equation*}
F:=\left\{ \left( \xi _{1},\xi _{2}\right) \in \mathbb{R}^{2}:\left\vert \xi
_{2}\right\vert \leq 1-\xi _{1}^{4},\;-1\leq \xi _{1}\leq 1\right\} .
\end{equation*}
For any arbitrary dual pair $\left( \xi ,\xi ^{\ast }\right) $, with $\xi
:=\left( \xi _{1},\xi _{2}\right) $, by the symmetry of $F$, just consider
the case $\xi _{2}\geq 0$ and $\xi _{1}\leq 0$. Using \cite[Example 8.3]{GP1}
we get:

\begin{enumerate}
\item[(i)] If $\xi _{2}>0$ then the (unique) normal vector $\xi ^{\ast }$ to 
$F$ at $\xi ,$ such that $\left\langle \xi ,\xi ^{\ast }\right\rangle =1$ is
given by%
\begin{equation*}
\xi ^{\ast }=\frac{1}{1+3\xi _{1}^{4}}\left( 4\xi _{1}^{3},1\right) .
\end{equation*}%
After a hard work, we obtained 
\begin{equation}
\hat{\varkappa}_{F}\left( \xi ,\xi ^{\ast }\right) =\frac{\hat{\gamma}%
_{F}\left( \xi ,\xi ^{\ast }\right) }{\left\Vert \xi ^{\ast }\right\Vert }%
\leq \frac{6\xi _{1}^{2}}{\sqrt{1+16\xi _{1}^{6}}}  \label{2.10}
\end{equation}%
and%
\begin{equation}
\hat{\varkappa}_{F}\left( \xi ,\xi ^{\ast }\right) \geq \frac{6\xi _{1}^{2}}{%
\sqrt{1+16\xi _{1}^{6}}~\Sigma ^{2}\left( \xi _{1}\right) }\text{,}
\label{Example3(3)}
\end{equation}%
where $\Sigma \left( \xi _{1}\right) :=\sqrt{1+\left( \underset{k=0}{\overset{3}{\sum }}\left\vert \xi _{1}\right\vert ^{k}\right) ^{2}}$. Combining the
estimates (\ref{2.10}) and (\ref{Example3(3)}) we see that the curvature $%
\hat{\varkappa}\left( \xi ,\xi ^{\ast }\right) $ is of order $O\left( \xi
_{1}^{2}\right) $ (as $\left\vert \xi _{1}\right\vert \rightarrow 0$). In
particular, $\hat{\varkappa}_{F}$ is equal to zero at the points $\left(
0,\pm 1\right) $.

\item[(ii)] If $\xi :=\left( -1,0\right) $ we have%
\begin{equation*}
\mathbf{N}_{F}\left( \xi \right) =\left\{ \left( v_{1},v_{2}\right) \in 
\mathbb{R}^{2}:v_{1}\leq -4\left\vert v_{2}\right\vert \right\} \text{.}
\end{equation*}%
For $\xi ^{\ast }\in \partial \mathbf{N}_{F}\left( \xi \right) $, by the
lower semicontinuity of the function $\left( \xi ,\xi ^{\ast }\right)
\mapsto \hat{\gamma}_{F}\left( \xi ,\xi ^{\ast }\right) $, we can apply the
same reasoning as above, but not for $\xi ^{\ast }\in \mbox{int}\mathbf{N}_{F}\left( \xi \right) $. In this last case we have $\hat{\varkappa}%
_{F}\left( \xi ,\xi ^{\ast }\right) =+\infty $ (see \cite[Proposition 3.8]%
{GP1}).
\end{enumerate}

\bigskip

Here we got only the estimates (\ref{2.10}) and (\ref{Example3(3)}), but
using the theory developed in this paper we will get an equality (see
Example \ref{Ex1}).

\bigskip

\section{Directional curvature\label{Directional}}

In everything that follows we consider a compact convex set $F\subset 
\mathbb{R}^{n}$, $n\geq 2$, with $\mathbf{0}\in \mbox{int}F$. Fixed $\xi
\in \partial F$ assume that there are $\delta >0$ and $f:\mathbb{R}^{n}%
\mathbf{\rightarrow }\mathbb{R}$ of class $\mathcal{C}^{2}$ at $\xi +\delta 
\mathbf{B}$ ($\mathbf{B}\subset \mathbb{R}^{n}$ represents the open unit
ball), such that%
\begin{equation*}
F\subset \left\{ x\in \mathbb{R}^{n}:f\left( x\right) \leq 0\right\} ,
\end{equation*}%
\begin{equation}
\left\langle \xi ,\nabla f\left( \xi \right) \right\rangle >0,  \label{5}
\end{equation}%
and such that, for $x\in \xi +\delta \mathbf{B}$, we have $x\in \partial F$
if and only if $f\left( x\right) =0$.

\begin{remark}
\label{r1}Thanks to (\ref{5}) and the continuity of $\nabla f\left( \cdot
\right) $ at $\xi $ there is $0<\delta ^{\prime }\leq \delta $ such that 
\begin{equation}
\inf_{\eta \in \xi +\delta ^{\prime }\mathbf{B}}~\left\langle \eta ,\nabla
f\left( \eta \right) \right\rangle >0.  \label{5'}
\end{equation}%
\bigskip In particular, we have $\nabla f\left( \eta \right) \neq \mathbf{0}$
for any $\eta \in \xi +\delta ^{\prime }\mathbf{B.}$
\end{remark}

\begin{proposition}
\label{ConeN}We have 
\begin{equation*}
\mathbf{N}_{F}\left( \eta \right) =\bigcup_{\lambda \geq 0}\lambda \nabla
f\left( \eta \right) ,\ \ \ \ \eta \in \partial F\cap \left( \xi +\delta
^{\prime }\mathbf{B}\right) .
\end{equation*}
\end{proposition}

\begin{proof}
By (\ref{15}), for an arbitrary $\eta \in \partial F\cap \left( \xi +\delta
^{\prime }\mathbf{B}\right)$, it's enough to prove that%
\begin{equation}
\mathbf{N}_{F}^{p}\left( \eta \right) =\bigcup_{\lambda \geq 0}\lambda
\nabla f\left( \eta \right) .  \label{14}
\end{equation}%
Since $f$ is of classe $\mathcal{C}^{2}$ at $\xi +\delta \mathbf{B}$, by 
\cite[Theorem 1.2.5 and Corolary 1.2.6 ]{CLSW}, there are $\sigma ,\rho >0$
such that $\left( \eta +\rho \mathbf{B}\right) \subset \left( \xi +\delta
^{\prime }\mathbf{B}\right) $ and 
\begin{equation*}
f\left( y\right) \geq f\left( \eta \right) +\left\langle \nabla f\left( \eta
\right) ,y-\eta \right\rangle -\sigma \left\Vert y-\eta \right\Vert ^{2}%
\text{,\ \ \ \ \ }\forall y\in \eta +\rho \mathbf{B},
\end{equation*}%
and consequently 
\begin{equation*}
\left\langle \nabla f\left( \eta \right) ,y-\eta \right\rangle \leq \sigma
\left\Vert y-\eta \right\Vert ^{2}\text{,\ \ \ \ \ }\forall y\in \left( \eta
+\rho \mathbf{B}\right) \cap F.
\end{equation*}%
Thanks to \cite[Proposition 1.1.5]{CLSW} $\nabla f\left( \eta \right) \in 
\mathbf{N}_{F}^{p}\left( \eta \right) .$ Since $\mathbf{N}_{F}^{p}\left(
\eta \right) $ is a cone we have, in fact, $\lambda \nabla f\left( \eta
\right) \in \mathbf{N}_{F}^{p}\left( \eta \right) $, $\lambda \geq 0$.

To prove the other inclusion at (\ref{14}) fix $\zeta \in \mathbf{N}_{F}^{p}\left( \eta \right)$. By \cite[Proposition 1.1.5]{CLSW} there is a
constant $\sigma >0$ such that 
\begin{equation*}
\left\langle \zeta ,y-\eta \right\rangle \leq \sigma \left\Vert y-\eta
\right\Vert ^{2},
\end{equation*}%
whenever $y$ belongs to $\partial F\cap \left( \xi +\delta \mathbf{B}\right)$. Put another way, this is equivalent to say that the point $\eta $
minimizes the function $y\mapsto \left\langle -\zeta ,y\right\rangle +\sigma
\left\Vert y-\eta \right\Vert ^{2}$ over all points $y$ satisfying $f\left(
y\right) =0$ and $\left\Vert y-\xi \right\Vert <\delta .$ The Lagrange
Multiplier Rule of classical calculus provides a scalar $\lambda \geq 0$ such
that $\zeta =\lambda \nabla f\left( \eta \right)$, which completes the
proof.
\end{proof}

\bigskip

Consequently, for any $\eta \in \partial F\cap \left( \xi +\delta ^{\prime }%
\mathbf{B}\right) $ fixed, $\mathfrak{J}_{F}^{-1}\left( \eta \right) =%
\mathbf{N}_{F}\left( \eta \right) \cap \partial F^{o}$ is a singleton, and
the unique $\eta ^{\ast }\in \mathfrak{J}_{F}^{-1}\left( \eta \right) $ is
given by 
\begin{equation}
\eta ^{\ast }=\frac{1}{\left\langle \eta ,\nabla f\left( \eta \right)
\right\rangle }\nabla f\left( \eta \right) .  \label{R2}
\end{equation}%
This means that $\eta ^{\ast }$ is well defined whenever $\eta $ is fixed.

On the other hand, $\nabla f\left( \eta \right) \neq $ $\mathbf{0}$ implies
that there is a first $i\in I:=\left\{ 1,...,n\right\} $ such that 
\begin{equation}
f_{x_{i}}\left( \eta \right) :=\frac{\partial f}{\partial x_{i}}\left( \eta
\right) \neq 0.  \label{fi}
\end{equation}%
Fixed such $i$ the hyperplane tangent to $F$ at $\eta $ is given by (see (%
\ref{PlanoT}))%
\begin{eqnarray*}
\mathbf{T}_{F}\left( \eta \right) &=&\left\{ v\in \mathbb{R}%
^{n}:\left\langle v,\nabla f\left( \eta \right) \right\rangle =0\right\} \\
&=&\left\{ \left( v_{1},...,v_{n}\right) \in \mathbb{R}^{n}:v_{i}=-%
\sum_{j=1,j\neq i}^{n}\frac{f_{x_{j}}\left( \eta \right) }{f_{x_{i}}\left(
\eta \right) }v_{j}\right\} .
\end{eqnarray*}%
Denote by $u^{j}\left( \eta \right)$, $j \in I\backslash \left\{ i\right\}$, the vector of $\mathbb{R}^{n}$ with $1$ in the $j$th coordinate, $-\frac{f_{x_{j}}\left( \eta \right)}{f_{x_{i}}\left(\eta \right)}$ in the $i$th
coordinate and $0$ in the others. Since $f$ is of class $\mathcal{C}^{2}$ at 
$\xi +\delta \mathbf{B}$, $u^{j}\left( \eta \right) $ will be close to $%
u^{j}\left( \xi \right) ,$ whenever $\eta $ is close to $\xi .$

\bigskip

For our results we need to introduce the following. Given $\eta \in \xi
+\delta ^{\prime }\mathbf{B}$ and $u\left( \eta \right) \in \mathbf{T}%
_{F}\left( \eta \right) $, $u\left( \eta \right) \neq \mathbf{0}$, consider
the subset of $\mathbb{R}^{n}$ 
\begin{equation*}
P\left( \eta ,u\left( \eta \right) \right) :=\mbox{span}\left\{ \nabla
f\left( \eta \right) ,u\left( \eta \right) \right\} +\eta ,
\end{equation*}%
where $\mbox{span}\left\{ \nabla f\left( \eta \right) ,u\left( \eta
\right) \right\} $ means the generated space by the vectors $\nabla f\left(
\eta \right) $ and $u\left( \eta \right) $. Note that the vectors $\nabla
f\left( \eta \right) $ and $u\left( \eta \right) $ are linearly independent,
so the set $P\left( \eta ,u\left( \eta \right) \right) $ is, in fact, a $2$%
-dimensional plane in $\mathbb{R}^{n}$ (it will simply be called a plane).

\bigskip

Bellow we introduce some directional notions, based on the respective notions
presented in \cite{GP1}, and already seen here in Section \ref{Exemplo}. To
simplify the notation, in general, we will not refer to the unique $\xi
^{\ast }$, given by (\ref{R2}).

\begin{definition}
\label{d1}The $\emph{2}$\emph{-dimensional modulus of strict convexity of }$%
F $\emph{\ at }$\xi \in \partial F$\emph{\ }(with respect to $\xi ^{\ast }$) 
\emph{in the direction of the vector} $u\left( \xi \right) \in \mathbf{T}%
_{F}\left( \xi \right) \backslash \left\{ \mathbf{0}\right\} $ is given by 
\begin{equation*}
\widehat{\mathfrak{C}}_{F}\left( r,\xi ,u\left( \xi \right) \right) =\inf
\left\{ \left\langle \xi -\eta ,\xi ^{\ast }\right\rangle :\eta \in F\cap
P\left( \xi ,u\left( \xi \right) \right) ,\left\Vert \xi -\eta \right\Vert
\geq r\right\} \text{, \ \ \ \ }r>0.
\end{equation*}%
The set $F$ is \emph{strictly convex at }$\xi $ \emph{in the direction of }$%
u\left( \xi \right) $ if $\ \widehat{\mathfrak{C}}_{F}\left( r,\xi ,u\left(
\xi \right) \right) >0$ for all $r>0.$
\end{definition}

\begin{remark}
\label{modulo}Since the set $F\subset \mathbb{R}^{n}$ is compact and convex
we have the equalities 
\begin{eqnarray*}
\widehat{\mathfrak{C}}_{F}\left( r,\xi ,u\left( \xi \right) \right) &=&\inf
\left\{ \left\langle \xi -\eta ,\xi ^{\ast }\right\rangle :\eta \in F\cap
P\left( \xi ,u\left( \xi \right) \right) ,\left\Vert \xi -\eta \right\Vert
=r\right\} \\
&=&\inf \left\{ \left\langle \xi -\eta ,\xi ^{\ast }\right\rangle :\eta \in
\partial F\cap P\left( \xi ,u\left( \xi \right) \right) ,\left\Vert \xi
-\eta \right\Vert =r\right\} ,
\end{eqnarray*}%
for any $r>0$ and $u\left( \xi \right) \in \mathbf{T}_{F}\left( \xi \right)
\backslash \left\{ \mathbf{0}\right\} .$
\end{remark}

\begin{proposition}
\label{single}Let $\xi \in \partial F$, $\xi^{\ast}\in \partial F^{o}$
given by (\ref{R2}) and $u\left( \xi \right) \in \mathbf{T}_{F}\left( \xi
\right) $\emph{, }$u\left( \xi \right) \neq \mathbf{0}.$ If $\widehat{%
\mathfrak{C}}_{F}\left( r,\xi ,u\left( \xi \right) \right) >0$ for all $r>0$
then $ \ \mathfrak{J}_{F}\left( \xi ^{\ast }\right) \cap P\left( \xi ,u\left(
\xi \right) \right) =\left\{ \xi \right\} .$
\end{proposition}

\begin{proof}
By construction $\xi \in \mathfrak{J}_{F}\left( \xi ^{\ast }\right) \cap
P\left( \xi ,u\left( \xi \right) \right) .$ If there was $\overline{\xi }\in 
\mathfrak{J}_{F}\left( \xi ^{\ast }\right) \cap P\left( \xi ,u\left( \xi
\right) \right) $ with $\overline{\xi }\neq \xi $, we would have 
\begin{equation*}
\left\langle \xi -\overline{\xi },\xi ^{\ast }\right\rangle =0,
\end{equation*}%
and consequently 
\begin{equation*}
\widehat{\mathfrak{C}}_{F}\left( r,\xi ,u\left( \xi \right) \right) =0,
\end{equation*}%
for $r:=\left\Vert \xi -\overline{\xi }\right\Vert >0,$ which is absurd.
\end{proof}

\begin{definition}
\label{def}The $\emph{2}$\emph{-dimensional curvature of }$F$\emph{\ at }$%
\xi \in \partial F$ (w.r.t. $\xi ^{\ast }$) \emph{in the direction of }$%
u^{j}\left( \xi \right) ,$ $j\in I\backslash \left\{ i\right\}$, is given by%
\begin{equation*}
\hat{\varkappa}_{F}\left( \xi ,u^{j}\left( \xi \right) \right) =\frac{1}{%
\left\Vert \xi ^{\ast }\right\Vert }\hat{\gamma}_{F}\left( \xi ,u^{j}\left(
\xi \right) \right) ,
\end{equation*}%
where%
\begin{equation*}
\hat{\gamma}_{F}\left( \xi ,u^{j}\left( \xi \right) \right) =\underset{ 
_{\substack{ \left( r,\eta \right) \rightarrow \left( 0^{+},\xi \right)  \\ %
\eta \in \partial F}}}{\lim \inf }\frac{\widehat{\mathfrak{C}}_{F}\left(
r,\eta ,u^{j}\left( \eta \right) \right) }{r^{2}}.
\end{equation*}

The set $F$ is said to be \emph{strictly convex of the second order at }$\xi 
$ \emph{in the direction of }$u^{j}\left( \xi \right) $ when $\hat{\varkappa}%
_{F}\left( \xi ,u^{j}\left( \xi \right) \right) >0.$
\end{definition}

\bigskip

Such at \cite[Proposition 3.7]{GP1} we may extend the concept of directional
strict convexity for the case of an arbitrary compact convex solid (do not
assuming that $\mathbf{0}\in \mbox{int}F$). For this, we need to remember
that the interior of any convex set $C$ in $\mathbb{R}^{n}$ relative to its
affine hull (the smallest affine set that includes $C$) is the \emph{%
relative interior} of $C$, denoted by $\mbox{rint}C$.

\begin{proposition}
\label{00}Let $\xi \in \partial F$, $i\in I$ as above, $j\in I\backslash
\left\{ i\right\} $, $y_{1},y_{2}\in \mbox{rint}\left( F\cap P\left( \xi
,u\left( \xi \right) \right) \right) $ and $\xi _{1}^{\ast }\in \mathfrak{J}%
_{F-y_{1}}^{-1}\left( \xi -y_{1}\right) $. Then there is an unique $\xi
_{2}^{\ast }\in \mathfrak{J}_{F-y_{2}}^{-1}\left( \xi -y_{2}\right) $
colinear with $\xi _{1}^{\ast }$ and such that%
\begin{equation}
\frac{1}{\left\Vert \xi _{1}^{\ast }\right\Vert }\hat{\gamma}%
_{F-y_{1}}\left( \xi -y_{1},u^{j}\left( \xi \right) \right) =\frac{1}{%
\left\Vert \xi _{2}^{\ast }\right\Vert }\hat{\gamma}_{F-y_{2}}\left( \xi
-y_{2},u^{j}\left( \xi \right) \right) .  \label{20}
\end{equation}
\end{proposition}

\begin{proof}
First, notice that $\xi _{1}^{\ast }$ is unique and, by (\ref{R2}), is given
by $\frac{1}{\left\langle \xi -y_{1},\nabla f\left( \xi \right)
\right\rangle }\nabla f\left( \xi \right) .$ As the same reason the unique $%
\xi _{2}^{\ast }\in \mathfrak{J}_{F-y_{2}}^{-1}\left( \xi -y_{2}\right) $ is
given by $\frac{1}{\left\langle \xi -y_{2},\nabla f\left( \xi \right)
\right\rangle }\nabla f\left( \xi \right) $, and it is colinear with $\xi
_{1}^{\ast }$.

Now, let us fix $\eta \in \partial F$ close to $\xi $, and the corresponding
vectors $\eta _{1}^{\ast }$ and $\eta _{2}^{\ast }$ (which are close to $\xi
_{1}^{\ast }$ and $\xi _{2}^{\ast }$, respectively). Notice that $\eta
_{1}^{\ast }\in \mathfrak{J}_{F-y_{1}}^{-1}\left( \eta -y_{1}\right) $
implies $\left\langle y-y_{1},\eta _{1}^{\ast }\right\rangle <1$ for any $%
y\in \mbox{int}F$, and we can write 
\begin{equation*}
\eta _{2}^{\ast }=\frac{1}{1+\left\langle y_{1}-y_{2},\eta _{1}^{\ast
}\right\rangle }\eta _{1}^{\ast }.
\end{equation*}%
So, from Definition \ref{d1}, we obtain%
\begin{eqnarray*}
&&\frac{1}{\left\Vert \eta _{2}^{\ast }\right\Vert }\widehat{\mathfrak{C}}%
_{F-y_{2}}\left( r,\eta -y_{2},u^{j}\left( \eta \right) \right) \\
&=&\frac{1}{\left\Vert \eta _{2}^{\ast }\right\Vert }\inf \left\{
\left\langle \eta -y_{2}-\zeta ,\eta _{2}^{\ast }\right\rangle :\zeta \in
\left( F\cap P\left( \eta ,u\left( \eta \right) \right) -y_{2}\right)
,\left\Vert \eta -y_{2}-\zeta \right\Vert \geq r\right\} \\
&=&\frac{1}{\left\Vert \eta _{2}^{\ast }\right\Vert }\inf \left\{
\left\langle \eta -y,\eta _{2}^{\ast }\right\rangle :y\in F\cap P\left( \eta
,u\left( \eta \right) \right) ,\left\Vert \eta -y\right\Vert \geq r\right\}
\\
&=&\frac{1}{\left\Vert \eta _{1}^{\ast }\right\Vert }\inf \left\{
\left\langle \eta -y,\eta _{1}^{\ast }\right\rangle :y\in F\cap P\left( \eta
,u\left( \eta \right) \right) ,\left\Vert \eta -y\right\Vert \geq r\right\}
\\
&=&\frac{1}{\left\Vert \eta _{1}^{\ast }\right\Vert }\widehat{\mathfrak{C}}%
_{F-y_{1}}\left( r,\eta -y_{1},u^{j}\left( \eta \right) \right) ,
\end{eqnarray*}%
i.e.,%
\begin{equation*}
\frac{1}{\left\Vert \eta _{2}^{\ast }\right\Vert }\widehat{\mathfrak{C}}%
_{F-y_{2}}\left( r,\eta -y_{2},u^{j}\left( \eta \right) \right) =\frac{1}{%
\left\Vert \eta _{1}^{\ast }\right\Vert }\widehat{\mathfrak{C}}%
_{F-y_{1}}\left( r,\eta -y_{1},u^{j}\left( \eta \right) \right)
\end{equation*}%
for all $r>0$. Dividing both parts of the last equality by $r^{2}$ and
passing to $\lim \inf $ as $r\rightarrow 0^{+}$, $\eta \rightarrow \xi $ we
easily come to (\ref{20}).
\end{proof}

\bigskip

In the last proof we used the known fact that $\mathbf{N}_{F-y_{i}}\left(
\xi -y_{i}\right) =\mathbf{N}_{F}\left( \xi \right) $.

\bigskip

Remember that $u\left( \xi \right) \in \mathbf{T}_{F}\left( \xi \right)
\backslash \left\{ \mathbf{0}\right\} $ if there are $n-1$ real numbers $%
\alpha _{j}, \ j\in I\backslash \left\{ i\right\} $, not simultaneously null,
such that%
\begin{equation*}
u\left( \xi \right) =\sum_{j=1,\ j\neq i}^{n}\alpha
_{j}u^{j}\left( \xi \right).
\end{equation*}%
This means that $u\left( \xi \right) $ is a vector of $\mathbb{R}^{n}$ with%
\begin{equation*}
-\sum_{j=1,\ j\neq i}^{n}\alpha _{j}\frac{f_{j}\left( \xi
\right) }{f_{i}\left( \xi \right) }
\end{equation*}%
at the $i$th coordinate and $\alpha _{j}$, $j\in I\backslash \left\{
i\right\} $, at the $j$th coordinate. If for any $\eta \in \partial F$ near $%
\xi $ we define $u\left( \eta \right) $ as the non-zero vector of $\mathbb{R}%
^{n}$ corresponding to $u\left( \xi \right) $, i.e., the $j$th coordinates, $%
j\in I\backslash \left\{ i\right\} $, are the same in both vectors, and the $%
i$th coordinate of $u\left( \eta \right) $ is given by%
\begin{equation*}
-\sum_{j=1, \ j\neq i}^{n}\alpha _{j}\frac{f_{j}\left( \eta
\right) }{f_{i}\left( \eta \right) },
\end{equation*}%
then it will be possible to put%
\begin{equation*}
\hat{\varkappa}_{F}\left( \xi ,u\left( \xi \right) \right) =\frac{1}{%
\left\Vert \xi ^{\ast }\right\Vert }\hat{\gamma}_{F}\left( \xi ,u\left( \xi
\right) \right) =\frac{1}{\left\Vert \xi ^{\ast }\right\Vert }\underset{ 
_{\substack{ \left( r,\eta \right) \rightarrow \left( 0^{+},\xi \right)  \\ %
\eta \in \partial F}}}{\lim \inf }\frac{\widehat{\mathfrak{C}}_{F}\left(
r,\eta ,u\left( \eta \right) \right) }{r^{2}}.
\end{equation*}%
Note that such vector $u\left( \eta \right) $ is, in fact, in $\mathbf{T}%
_{F}\left( \eta \right) .$

\begin{proposition}
Let $\xi \in \partial F$. If there is $u\left( \xi \right) \in \mathbf{T}%
_{F}\left( \xi \right) \backslash \left\{ \mathbf{0}\right\} $ such that $%
\hat{\gamma}_{F}\left( \xi ,u\left( \xi \right) \right) >0$, then we will
have $\mathfrak{J}_{F}\left( \eta ^{\ast }\right) \cap P\left( \eta ,u\left(
\eta \right) \right) =\left\{ \eta \right\} $ for every $\eta $ close enough
to $\xi $ (and respective $\eta ^{\ast }$ given by (\ref{R2})).
\end{proposition}

\begin{proof}
The condition $\hat{\gamma}_{F}\left( \xi ,u\left( \xi \right) \right) >0$
means that for some $\theta >0$ and $\rho >0$ the inequality%
\begin{equation}
\widehat{\mathfrak{C}}_{F}\left( r,\eta ,u\left( \eta \right) \right) \geq
\theta r^{2}  \label{17}
\end{equation}%
takes place whenever $\left\Vert \xi -\eta \right\Vert \leq \rho $, $\eta
\in \partial F$ and $0<r\leq \rho $. Thanks to the monotony of the
function $r\mapsto \widehat{\mathfrak{C}}_{F}\left( r,\eta ,u\left( \eta
\right) \right) $, decreasing if necessary the constant $\theta >0,$ we can
assume that (\ref{17}) is valid for all positive $r$. In fact, $\widehat{%
\mathfrak{C}}_{F}\left( r,\eta ,u\left( \eta \right) \right) =+\infty $
whenever $r>2\left\Vert F\right\Vert $ and for $\rho \leq r\leq 2\left\Vert
F\right\Vert $ we have 
\begin{equation*}
\widehat{\mathfrak{C}}_{F}\left( r,\eta ,u\left( \eta \right) \right) \geq 
\widehat{\mathfrak{C}}_{F}\left( \rho ,\eta ,u\left( \eta \right) \right)
\geq \theta \left( \frac{\rho }{r}\right) ^{2}r^{2}\geq \theta \left( \frac{%
\rho }{2\left\Vert F\right\Vert }\right) ^{2}r^{2}.
\end{equation*}%
Hence, $\widehat{\mathfrak{C}}_{F}\left( r,\eta ,u\left( \eta \right)
\right) >0$, for all $r>0$, and the conclusion follows from Proposition \ref%
{single}.
\end{proof}

\bigskip

In the next proposition $\hat{\gamma}_{F}\left( \xi \right) $ represents $%
\hat{\gamma}_{F}\left( \xi ,\xi ^{\ast }\right) $, given by (\ref{gama}),
for the unique $\xi ^{\ast }=\frac{1}{\left\langle \xi ,\nabla f\left( \xi
\right) \right\rangle }\nabla f\left( \xi \right) .$

\begin{proposition}
\label{curvaturas}We have%
\begin{equation}
\hat{\gamma}_{F}\left( \xi ,u\left( \xi \right) \right) \geq \hat{\gamma}%
_{F}\left( \xi \right) ,\ \ \ \ \ \forall u\left( \xi \right) \in \mathbf{T}%
_{F}\left( \xi \right) \backslash \left\{ \mathbf{0}\right\} .  \label{13}
\end{equation}%
Furthermore, if $\hat{\gamma}_{F}\left( \xi ,u\left( \xi \right) \right) =0$
for some $u\left( \xi \right) \in \mathbf{T}_{F}\left( \xi \right)
\backslash \left\{ \mathbf{0}\right\}$, we have $\hat{\gamma}_{F}\left( \xi
\right) =0$ too.
\end{proposition}

\begin{proof}
In fact, by (\ref{Def3.2}),%
\begin{eqnarray}
\underset{_{\substack{ \left( r,\eta \right) \rightarrow \left( 0^{+},\xi
\right)  \\ \eta \in \partial F}}}{\lim \inf }\frac{\widehat{\mathfrak{C}}%
_{F}\left( r,\eta ,u\left( \eta \right) \right) }{r^{2}} &=&\underset{ 
_{\substack{ \left( r,\eta \right) \rightarrow \left( 0^{+},\xi \right)  \\ %
\eta \in \partial F, \ \eta ^{\ast }=\frac{1}{\left\langle \eta ,\nabla f\left(
\eta \right) \right\rangle }\nabla f\left( \eta \right) }}}{\lim \inf }\frac{%
\widehat{\mathfrak{C}}_{F}\left( r,\eta ,u\left( \eta \right) \right) }{r^{2}%
}  \notag \\
&\geq &\underset{_{\substack{ \left( r,\eta \right) \rightarrow \left(
0^{+},\xi \right)  \\ \eta \in \partial F, \ \eta ^{\ast }=\frac{1}{%
\left\langle \eta ,\nabla f\left( \eta \right) \right\rangle }\nabla f\left(
\eta \right) }}}{\lim \inf }\frac{\widehat{\mathfrak{C}}_{F}\left( r,\eta
,\eta ^{\ast }\right) }{r^{2}}  \label{16} \\
&\geq &\underset{_{\substack{ \left( r,\eta ,\eta ^{\ast }\right)
\rightarrow \left( 0^{+},\xi ,\xi ^{\ast }\right)  \\ \eta \in \mathfrak{J}%
_{F}\left( \eta ^{\ast }\right), \ \eta ^{\ast }\in \partial F^{o}}}}{\lim
\inf }\frac{\widehat{\mathfrak{C}}_{F}\left( r,\eta ,\eta ^{\ast }\right) }{%
r^{2}},  \label{2}
\end{eqnarray}%
which implies (\ref{13}).

Now, recalling that we always have $\hat{\gamma}_{F}\left( \xi \right) \geq
0 $, if there is $u\left( \xi \right) \in \mathbf{T}_{F}\left( \xi \right)
\backslash \left\{ \mathbf{0}\right\} $ such that $\hat{\gamma}_{F}\left(
\xi ,u\left( \xi \right) \right) =0$, we will have $\hat{\gamma}_{F}\left(
\xi \right) =0.$
\end{proof}

\bigskip

When $n=2$, we have $\mathbf{T}_{F}\left( \xi \right) =\mbox{span}\left\{
t_{F}\left( \xi \right) \right\} =\left\{ \lambda t_{F}\left( \xi \right)
:\lambda \in \mathbb{R}\right\}$, for $\ t_{F}\left( \xi \right) :=\left(
-f_{2}\left( \xi \right) ,f_{1}\left( \xi \right) \right) $, and the plane $%
P\left( \xi ,\lambda t_{F}\left( \xi \right) \right) $ coincide with $%
\mathbb{R}^{2}$ for every $\lambda \in \mathbb{R}\backslash \left\{
0\right\} $.

\begin{proposition}
\label{n2}For $n=2$ the equality holds at (\ref{13}) if $\ \hat{\gamma}%
_{F}\left( \xi ,\lambda t_{F}\left( \xi \right) \right) >0$, for some $%
\lambda \in \mathbb{R}\backslash \left\{ 0\right\} $.
\end{proposition}

\begin{proof}
Fixed $\eta \in \left( \xi +\delta ^{\prime }\mathbf{B}\right) \cap \partial
F$ and $u\left( \eta \right) \in \mathbf{T}_{F}\left( \eta \right)
\backslash \left\{ \mathbf{0}\right\} $, as defined above, we have $P\left(
\eta ,u\left( \eta \right) \right) =\mathbb{R}^{2}$, which implies that 
\begin{equation*}
\widehat{\mathfrak{C}}_{F}\left( r,\eta ,u\left( \eta \right) \right) =%
\widehat{\mathfrak{C}}_{F}\left( r,\eta ,\eta ^{\ast }\right) \text{,\ \ \ }%
\forall r>0
\end{equation*}%
(see Definition \ref{d1} and (\ref{Def3.2})). Therefore there is an equality
in (\ref{16}).

If $\lambda \in \mathbb{R}\backslash \left\{ 0\right\} $ is such that $\hat{%
\gamma}_{F}\left( \xi ,\lambda t_{F}\left( \xi \right) \right) >0$, then
there exist $\theta >0$ and $\varepsilon >0$ such that%
\begin{equation*}
\widehat{\mathfrak{C}}_{F}\left( r,\eta ,\lambda t_{F}\left( \eta \right)
\right) \geq \theta r^{2},
\end{equation*}%
whenever\ $\left\Vert \eta -\xi \right\Vert \leq \varepsilon $, $\eta \in
\partial F$ and $0<r<\varepsilon .$ Now, by Proposition \ref{single}, for
such $\eta $ and respectives $t_{F}\left( \eta \right) $ and $\eta ^{\ast }$, we have $\ \mathfrak{J}_{F}\left( \eta ^{\ast }\right) =\left\{ \eta
\right\} $. Therefore, we obtain an equality at (\ref{2}), and consequently%
\begin{equation}
\hat{\gamma}_{F}\left( \xi ,\lambda t_{F}\left( \xi \right) \right) =\hat{%
\gamma}_{F}\left( \xi \right) .  \label{18}
\end{equation}
\end{proof}

\bigskip

The last proposition together with Proposition \ref{curvaturas} imply that $%
\hat{\gamma}_{F}\left( \xi \right) =\hat{\gamma}_{F}\left( \xi ,u\left( \xi
\right) \right) ,$ for every $u\left( \xi \right) \in \mathbf{T}_{F}\left(
\xi \right) .$ The last conclusion was already expected, since in $\mathbb{R}%
^{2}$ there is only one tangent direction.

\section{The main result}

In this section we prove that, under our conditions, the directional
curvature $\hat{\varkappa}_{F}\left( \xi ,u^{j}\left( \xi \right) \right) ,$ 
$j\in I\backslash \left\{ i\right\} $, can be calculated very easily. For
this we need to compute the second derivative of $f$ at $\xi ,$ $\nabla
^{2}f\left( \xi \right) $, given, as usual, by the $n\times n$ matrix with $%
\frac{\partial ^{2}f}{\partial x_{r}\partial x_{s}}\left( \xi \right) $ at
the row $r$ and column $s$, for every $r,s\in \left\{ 1,...,n\right\} .$

\begin{theorem}
\label{Teo1}Let a compact convex set $F\subset \mathbb{R}^{n}$, $n\geq 2$,
with $\mathbf{0}\in \mathbb{R}^{n}$ in its interior, and a point $\xi \in
\partial F.$ Assume that there are $\delta >0$ and $f:\mathbb{R}%
^{n}\rightarrow \mathbb{R}$ of class $\mathcal{C}^{2}$ at $\xi +\delta 
\mathbf{B}$, such that 
\begin{equation*}
F\subset \left\{ x\in \mathbb{R}^{n}:f\left( x\right) \leq 0\right\} ,
\end{equation*}%
\begin{equation*}
\left\langle \xi ,\nabla f\left( \xi \right) \right\rangle >0,
\end{equation*}%
and such that, for $x\in \xi +\delta \mathbf{B}$, we have $%
x\in \partial F$ if and only if $f\left( x\right) =0.$ Then we have%
\begin{equation}
\hat{\gamma}_{F}\left( \xi ,u^{j}\left( \xi \right) \right) =\frac{1}{%
2\left\langle \xi ,\nabla f\left( \xi \right) \right\rangle \left\Vert
u^{j}\left( \xi \right) \right\Vert ^{2}}\left\langle \nabla ^{2}f\left( \xi
\right) u^{j}\left( \xi \right) ,u^{j}\left( \xi \right) \right\rangle ,%
\text{ \ \ \ }j\in I\backslash \left\{ i\right\} .  \label{1}
\end{equation}
\end{theorem}

\begin{proof}
	By hypothesis $\left\langle \xi ,\nabla f\left( \xi \right) \right\rangle >0$, so let us fix the first $i\in I$ such that $f_{x_{i}}\left( \xi \right)\neq 0$ (see Remark \ref{r1}). For any $\eta \in \mathbb{R}^{n}$, let $\eta^{i}\in \mathbb{R}^{n-1}$ the vector $\eta $ without the $i$th coordinate. Thanks to the Implicit Function Theorem there are a neighbourhood $U:=\xi^{i}+\delta _{1}\mathbf{B}\subset \mathbb{R}^{n-1}$, $0<\delta_{1}\leq\delta$, and a $\mathcal{C}^{2}$ function $g:U\rightarrow \mathbb{R}$ such that:
	\begin{enumerate}
		\item[(i)] $f\left( \eta ^{i},g\left( \eta ^{i}\right) \right) =0$, for any $\eta ^{i}\in U$,
		\item[(ii)] for $\eta ^{i}\in U$ such that $f\left( \eta \right) =0$ we have $\eta _{i}=g\left( \eta ^{i}\right)$, and 
		\item[(iii)] for any $\eta ^{i}\in U$ we have
		\begin{equation*}
			\frac{\partial g}{\partial x_{j}}\left( \eta ^{i}\right) =-\frac{\frac{\partial f}{\partial x_{j}}\left( \eta ^{i},g\left( \eta ^{i}\right) \right)}{\frac{\partial f}{\partial x_{i}}\left(\eta^{i},g\left( \eta ^{i}\right)\right) }\text{ \ \ }(j\in I\backslash\left\{ i\right\}),
		\end{equation*}
	where $\left( \eta ^{i},g\left( \eta ^{i}\right) \right) \in \mathbb{R}^{n}$ represents the vector $\eta $ with $g\left( \eta ^{i}\right) $ instead of $\eta _{i}$.
\end{enumerate}
Let $j\in I\backslash \left\{ i\right\} $. By (\ref{5'}) for each $\varepsilon >0$ there is $0<\overline{\delta }=\overline{\delta }\left(
\varepsilon \right) \leq \min \left\{ \delta ^{\prime },\delta _{1}\right\} $
such that%
\begin{equation}
\left\Vert \frac{1}{\left\langle \eta ,\nabla f\left( \eta \right)
\right\rangle }\nabla f\left( \eta \right) -\frac{1}{\left\langle \xi
,\nabla f\left( \xi \right) \right\rangle }\nabla f\left( \xi \right)
\right\Vert <\varepsilon  \label{R4}
\end{equation}%
holds for any $\eta \in \xi +\overline{\delta }\mathbf{B}$ (by the
continuity of $\nabla f\left( \cdot \right) $ at $\xi $),%
\begin{equation}
\left\vert \frac{\left\langle \nabla ^{2}f\left( \zeta \right)
v,v\right\rangle }{\left\langle \eta ,\nabla f\left( \eta \right)
\right\rangle }-\frac{\left\langle \nabla ^{2}f\left( \xi \right)
v,v\right\rangle }{\left\langle \xi ,\nabla f\left( \xi \right)
\right\rangle }\right\vert <\frac{\varepsilon }{2},  \label{6}
\end{equation}%
for any $\eta ,\zeta \in \xi +\overline{\delta }\mathbf{B}$ and any $v\in 
\mathbb{R}^{n}$, $\left\Vert v\right\Vert =1$ (by the continuity of $\nabla
^{2}f\left( \cdot \right) $ at $\xi $), and such that 
\begin{equation}
\left\vert \left\langle \nabla ^{2}f\left( \xi \right) \frac{y-\eta }{%
\left\Vert y-\eta \right\Vert },\frac{y-\eta }{\left\Vert y-\eta \right\Vert 
}\right\rangle -\left\langle \nabla ^{2}f\left( \xi \right) \frac{%
u^{j}\left( \xi \right) }{\left\Vert u^{j}\left( \xi \right) \right\Vert },%
\frac{u^{j}\left( \xi \right) }{\left\Vert u^{j}\left( \xi \right)
\right\Vert }\right\rangle \right\vert <\left\langle \xi ,\nabla f\left( \xi
\right) \right\rangle \varepsilon ,  \label{8}
\end{equation}%
holds for any $\eta ,y\in \xi +\overline{\delta }\mathbf{B}$, $y\in P\left(
\eta ,u^{j}\left( \eta \right) \right) $, $y\neq \eta $, with $f\left(
y\right) =f\left( \eta \right) =0$ (using the continuity of $\nabla g\left(
\cdot \right) $ and $\nabla f\left( \cdot \right) $ at $\xi ^{i}$ and $\xi $%
, respectively, and using the Lagrange Mean Value Theorem).

\medskip

Let us prove the inequality $"\geq "$ in (\ref{1}), assuming that $\hat{%
\gamma}_{F}\left( \xi ,u^{j}\left( \xi \right) \right) <+\infty $, because
in the other case there is nothing to prove.

Let us fix $\varepsilon >0$, the corresponding $\overline{\delta }>0$, $0<r<%
\frac{\overline{\delta }}{2}$ and $\eta \in \left( \xi +\frac{\overline{%
\delta }}{2}\mathbf{B}\right) \cap \partial F$. We want to prove that 
\begin{equation*}
\frac{\widehat{\mathfrak{C}}_{F}\left( r,\eta ,u^{j}\left( \eta \right)
\right) }{r^{2}}>\frac{1}{2\left\langle \xi ,\nabla f\left( \xi \right)
\right\rangle \left\Vert u^{j}\left( \xi \right) \right\Vert ^{2}}%
\left\langle \nabla ^{2}f\left( \xi \right) u^{j}\left( \xi \right)
,u^{j}\left( \xi \right) \right\rangle -\varepsilon .
\end{equation*}%
Remember that $\eta ^{\ast }=\frac{1}{\left\langle \eta ,\nabla f\left( \eta
\right) \right\rangle }\nabla f\left( \eta \right) $. By Remark \ref{modulo}
there is $y\in \partial F\cap P\left( \eta ,u^{j}\left( \eta \right) \right) 
$ with $\left\Vert \eta -y\right\Vert =r$, such that 
\begin{equation}
\widehat{\mathfrak{C}}_{F}\left( r,\eta ,u^{j}\left( \eta \right) \right)
>\left\langle \eta -y,\eta ^{\ast }\right\rangle -\frac{\varepsilon }{4}%
r^{2}.  \label{4}
\end{equation}%
Notice that $\left\Vert \eta -y\right\Vert =r>0$ implies $y\neq \eta $.
Putting%
\begin{equation}
v:=\frac{y-\eta }{\left\Vert y-\eta \right\Vert },  \label{9}
\end{equation}%
then $\eta +rv=y$. Thanks to the Taylor's formula (see, e.g., \cite[p.75]%
{Berger}) 
\begin{equation*}
f\left( \eta +rv\right) =f\left( \eta \right) +\left\langle rv,\nabla
f\left( \eta \right) \right\rangle +\int_{0}^{r}\left\langle \nabla
^{2}f\left( \eta +\tau v\right) v,v\right\rangle \left( r-\tau \right) d\tau,
\end{equation*}%
and by the definition of $v$ 
\begin{equation*}
f\left( y\right) =f\left( \eta \right) +\left\langle y-\eta ,\nabla f\left(
\eta \right) \right\rangle +\int_{0}^{r}\left\langle \nabla ^{2}f\left( \eta
+\tau v\right) v,v\right\rangle \left( r-\tau \right) d\tau .
\end{equation*}%
Hence, by using the Mean Value Theorem for integrals and remembering that $%
f\left( \eta \right) =f\left( y\right) =0$, we obtain%
\begin{eqnarray}
\left\langle \eta -y,\nabla f\left( \eta \right) \right\rangle
&=&\int_{0}^{r}\left\langle \nabla ^{2}f\left( \eta +\tau v\right)
v,v\right\rangle \left( r-\tau \right) d\tau  \notag \\
&=&\frac{r^{2}}{2}\left\langle \nabla ^{2}f\left( \eta +\tau v\right)
v,v\right\rangle ,  \label{3}
\end{eqnarray}%
for some $\tau =\tau \left( r,v\right) \in \left] 0,r\right[ $. Let us fix
such $\tau $. By (\ref{4}), (\ref{R2}) and (\ref{3}), respectively, we have%
\begin{eqnarray}
\widehat{\mathfrak{C}}_{F}\left( r,\eta ,u^{j}\left( \eta \right) \right)
&>&\left\langle \eta -y,\eta ^{\ast }\right\rangle -\frac{\varepsilon }{4}%
r^{2}  \notag \\
&=&\frac{1}{\left\langle \eta ,\nabla f\left( \eta \right) \right\rangle }%
\left\langle \eta -y,\nabla f\left( \eta \right) \right\rangle -\frac{%
\varepsilon }{4}r^{2}  \notag \\
&\geq &\frac{r^{2}}{2\left\langle \eta ,\nabla f\left( \eta \right)
\right\rangle }\left\langle \nabla ^{2}f\left( \eta +\tau v\right)
v,v\right\rangle -\frac{\varepsilon }{4}r^{2}.  \label{7}
\end{eqnarray}%
Since 
\begin{equation}
\left\Vert \eta +\tau v-\xi \right\Vert =\left\Vert \eta +\tau \frac{y-\eta 
}{\left\Vert y-\eta \right\Vert }-\xi \right\Vert \leq \left\Vert \eta -\xi
\right\Vert +\tau <\overline{\delta },  \label{10}
\end{equation}%
by (\ref{6}) we obtain%
\begin{equation*}
\left\vert \frac{\left\langle \nabla ^{2}f\left( \eta +\tau v\right)
v,v\right\rangle }{\left\langle \eta ,\nabla f\left( \eta \right)
\right\rangle }-\frac{\left\langle \nabla ^{2}f\left( \xi \right)
v,v\right\rangle }{\left\langle \xi ,\nabla f\left( \xi \right)
\right\rangle }\right\vert <\frac{\varepsilon }{2},
\end{equation*}%
and using (\ref{8}) and (\ref{7}) we conclude 
\begin{eqnarray*}
\frac{\widehat{\mathfrak{C}}_{F}\left( r,\eta ,u^{j}\left( \eta \right)
\right) }{r^{2}} &>&\frac{1}{2\left\langle \eta ,\nabla f\left( \eta \right)
\right\rangle }\left\langle \nabla ^{2}f\left( \eta +\tau v\right)
v,v\right\rangle -\frac{\varepsilon }{4} \\
&>&\frac{1}{2\left\langle \xi ,\nabla f\left( \xi \right) \right\rangle }%
\left\langle \nabla ^{2}f\left( \xi \right) v,v\right\rangle -\frac{%
\varepsilon }{2} \\
&>&\frac{1}{2\left\langle \xi ,\nabla f\left( \xi \right) \right\rangle }%
\left\langle \nabla ^{2}f\left( \xi \right) \frac{u^{j}\left( \xi \right) }{%
\left\Vert u^{j}\left( \xi \right) \right\Vert },\frac{u^{j}\left( \xi
\right) }{\left\Vert u^{j}\left( \xi \right) \right\Vert }\right\rangle
-\varepsilon .
\end{eqnarray*}%
Passing to the limit as $\varepsilon \rightarrow 0^{+}$ we obtain the
desired inequality: 
\begin{equation*}
\hat{\gamma}_{F}\left( \xi ,u^{j}\left( \xi \right) \right) \geq \frac{1}{%
2\left\langle \xi ,\nabla f\left( \xi \right) \right\rangle \left\Vert
u^{j}\left( \xi \right) \right\Vert ^{2}}\left\langle \nabla ^{2}f\left( \xi
\right) u^{j}\left( \xi \right) ,u^{j}\left( \xi \right) \right\rangle .
\end{equation*}

\medskip

In order to show the opposite inequality let us assume that $\hat{\gamma}%
_{F}\left( \xi ,u^{j}\left( \xi \right) \right) >0$ (the case $\hat{\gamma}%
_{F}\left( \xi ,u^{j}\left( \xi \right) \right) =0$ is trivial).

Let us fix $\varepsilon >0$ and $0<\overline{\delta }=\overline{\delta }%
\left( \varepsilon \right) \leq \min \left\{ \delta ^{\prime },\delta
_{1}\right\} $ such that (\ref{R4}), (\ref{6}), (\ref{8}) and%
\begin{equation}
\frac{\widehat{\mathfrak{C}}_{F}\left( r,\eta ,u^{j}\left( \eta \right)
\right) }{r^{2}}>\hat{\gamma}_{F}\left( \xi ,u^{j}\left( \xi \right) \right)
-\frac{\varepsilon }{4}  \label{12}
\end{equation}%
holds for every $0<r<\overline{\delta }$ and $\eta \in \partial F$ with $%
\left\Vert \eta -\xi \right\Vert <\overline{\delta }$.

Let us fix $0<r<\frac{\overline{\delta }}{2}$, $\eta \in \left( \xi +\frac{%
\overline{\delta }}{2}\mathbf{B}\right) \cap \partial F$ and respective $%
\eta ^{\ast }$. We have
\begin{eqnarray}
\widehat{\mathfrak{C}}_{F}\left( r,\eta ,u^{j}\left( \eta \right) \right)
&=&\inf \left\{ \left\langle \eta -y,\eta ^{\ast }\right\rangle :y\in
\partial F\cap P\left( \eta ,u^{j}\left( \eta \right) \right) ,\left\Vert
\eta -y\right\Vert =r\right\}  \notag \\
&=&\frac{1}{\left\langle \eta ,\nabla f\left( \eta \right) \right\rangle }
\inf \left\{ \left\langle \eta -y,\nabla f\left( \eta \right) \right\rangle
:y\in \partial F\cap P\left( \eta ,u^{j}\left( \eta \right) \right),\left\Vert \eta -y\right\Vert =r\right\}.   \notag \\
\label{11}
\end{eqnarray}%
Now let us fix $y\in \partial F\cap P\left( \eta ,u^{j}\left( \eta \right)
\right) $ with $\left\Vert \eta -y\right\Vert =r$ and define $v$ as in (\ref%
{9}). Proceeding as above we obtain (\ref{3}) for some $\tau =\tau \left(
r,\eta \right) \in \left] 0,r\right[ $. Let us fix this $\tau $. Using (\ref%
{3}), (\ref{10}), (\ref{6}) and (\ref{8}), respectively, we obtain%
\begin{eqnarray*}
\frac{1}{r^{2}}\frac{1}{\left\langle \eta ,\nabla f\left( \eta \right)
\right\rangle }\left\langle \eta -y,\nabla f\left( \eta \right)
\right\rangle &=&\frac{1}{2\left\langle \eta ,\nabla f\left( \eta \right)
\right\rangle }\left\langle \nabla ^{2}f\left( \eta +\tau v\right)
v,v\right\rangle \\
&<&\frac{1}{2\left\langle \xi ,\nabla f\left( \xi \right) \right\rangle }%
\left\langle \nabla ^{2}f\left( \xi \right) v,v\right\rangle +\frac{%
\varepsilon }{4} \\
&<&\frac{1}{2\left\langle \xi ,\nabla f\left( \xi \right) \right\rangle }%
\left\langle \nabla ^{2}f\left( \xi \right) \frac{u^{j}\left( \xi \right) }{%
\left\Vert u^{j}\left( \xi \right) \right\Vert },\frac{u^{j}\left( \xi
\right) }{\left\Vert u^{j}\left( \xi \right) \right\Vert }\right\rangle +%
\frac{\varepsilon }{2}+\frac{\varepsilon }{4} \\
&=&\frac{1}{2\left\langle \xi ,\nabla f\left( \xi \right) \right\rangle }%
\frac{1}{\left\Vert u^{j}\left( \xi \right) \right\Vert ^{2}}\left\langle
\nabla ^{2}f\left( \xi \right) u^{j}\left( \xi \right) ,u^{j}\left( \xi
\right) \right\rangle +\frac{3\varepsilon }{4}.
\end{eqnarray*}%
Consequently (see (\ref{11}))%
\begin{equation*}
\frac{\widehat{\mathfrak{C}}_{F}\left( r,\eta ,u^{j}\left( \eta \right)
\right) }{r^{2}}<\frac{1}{2\left\langle \xi ,\nabla f\left( \xi \right)
\right\rangle }\frac{1}{\left\Vert u^{j}\left( \xi \right) \right\Vert ^{2}}%
\left\langle \nabla ^{2}f\left( \xi \right) u^{j}\left( \xi \right)
,u^{j}\left( \xi \right) \right\rangle +\frac{3\varepsilon }{4},
\end{equation*}%
\ and by (\ref{12})%
\begin{equation*}
\hat{\gamma}_{F}\left( \xi ,u^{j}\left( \xi \right) \right) <\frac{1}{%
2\left\langle \xi ,\nabla f\left( \xi \right) \right\rangle }\frac{1}{%
\left\Vert u^{j}\left( \xi \right) \right\Vert ^{2}}\left\langle \nabla
^{2}f\left( \xi \right) u^{j}\left( \xi \right) ,u^{j}\left( \xi \right)
\right\rangle +\varepsilon .
\end{equation*}%
Passing to the limit as $\varepsilon \rightarrow 0^{+}$ we obtain the
inequality $"\leq "$ in (\ref{1}).
\end{proof}

\bigskip

Remembering the Definition \ref{def} we have%
\begin{equation}
\hat{\varkappa}_{F}\left( \xi ,u^{j}\left( \xi \right) \right) =\frac{1}{%
2\left\Vert \nabla f\left( \xi \right) \right\Vert \left\Vert u^{j}\left(
\xi \right) \right\Vert ^{2}}\left\langle \nabla ^{2}f\left( \xi \right)
u^{j}\left( \xi \right) ,u^{j}\left( \xi \right) \right\rangle \text{, }\ \
j\in I\backslash \left\{ i\right\} .  \label{R3}
\end{equation}

Note that, for a fixed $j\in I\backslash \left\{ i\right\} $ and $\lambda
\in \mathbb{R}\backslash \left\{ 0\right\} $ we have%
\begin{equation}
\hat{\varkappa}_{F}\left( \xi ,\lambda u^{j}\left( \xi \right) \right) =%
\frac{1}{2\left\Vert \nabla f\left( \xi \right) \right\Vert \left\Vert
u^{j}\left( \xi \right) \right\Vert ^{2}}\left\langle \nabla ^{2}f\left( \xi
\right) u^{j}\left( \xi \right) ,u^{j}\left( \xi \right) \right\rangle ,
\label{Cn2}
\end{equation}%
as it would be expected. Moreover, following the proof of Theorem \ref{Teo1}
it is possible to prove that

\begin{corollary}
\label{CTeo1}We have%
\begin{equation*}
\hat{\varkappa}_{F}\left( \xi ,u\left( \xi \right) \right) =\frac{1}{%
2\left\Vert \nabla f\left( \xi \right) \right\Vert \left\Vert u\left( \xi
\right) \right\Vert ^{2}}\left\langle \nabla ^{2}f\left( \xi \right) u\left(
\xi \right) ,u\left( \xi \right) \right\rangle ,
\end{equation*}%
for any $u\left( \xi \right) \in \mathbf{T}_{F}\left( \xi \right) \backslash
\left\{ \mathbf{0}\right\} $.
\end{corollary}

Notice that, by (\ref{Cn2}), for $n=2,$ to say that $\hat{\varkappa}%
_{F}\left( \xi ,u\left( \xi \right) \right) >0$ for some $u\left( \xi
\right) \in \mathbf{T}_{F}\left( \xi \right) \backslash \left\{ \mathbf{0}%
\right\} $, is the same as saying that $\hat{\varkappa}_{F}\left( \xi
,\left( -f_{2}\left( \xi \right) ,f_{1}\left( \xi \right) \right) \right) >0$%
. Therefore, by Proposition \ref{n2}, if there is $\lambda \in \mathbb{R}%
\backslash \left\{ 0\right\} $ such that $\hat{\varkappa}_{F}\left( \xi
,\lambda \left( -f_{2}\left( \xi \right) ,f_{1}\left( \xi \right) \right)
\right) >0$, we will have 
\begin{equation*}
\hat{\varkappa}_{F}\left( \xi ,u\left( \xi \right) \right) =\hat{\varkappa}%
_{F}\left( \xi \right) ,\ \ \ \ \forall u\left( \xi \right) \in \mathbf{T}%
_{F}\left( \xi \right) \backslash \left\{ \mathbf{0}\right\} .
\end{equation*}

\bigskip

Since $f$ is of class $\mathcal{C}^{2}$ in $\xi +\delta \mathbf{B,}$ all the
conclusions will remain valid if we replace $\xi $ with any $\eta \in
\partial F\cap \left( \xi +\delta \mathbf{B}\right) .$

\bigskip

Following the idea of G. Crasta and A. Malusa presented in \cite[pg.5749]%
{CrastaMalusa}, we have the following result.

\begin{theorem}
Let $\xi \in \partial F$ and $\hat{\varkappa}_{F}\left( \xi ,u^{j_{1}}\left(
\xi \right) \right) \leq \cdots \leq \hat{\varkappa}_{F}\left( \xi
,u^{j_{n-1}}\left( \xi \right) \right) $ be the curvatures of $F$ at $\xi $
in the direction of the $n-1$ vectors that generate $\mathbf{T}_{F}\left(
\xi \right) $. If 
\begin{equation}
\left\Vert \nabla ^{2}f\left( \xi \right) \right\Vert :=\sup_{\substack{ %
u,v\in \mathbb{R}^{n}  \\ \left\Vert u\right\Vert =\left\Vert v\right\Vert
=1 }}\left\vert \left\langle \nabla ^{2}f\left( \xi \right) u,v\right\rangle
\right\vert <\infty ,  \label{Norma}
\end{equation}%
then%
\begin{equation*}
\hat{\varkappa}_{F}\left( \xi ,u^{j_{1}}\left( \xi \right) \right)
=\min_{u\in U_{\xi }}\hat{\varkappa}\left( u\right) \text{ \ \ and }\ \ \hat{%
\varkappa}_{F}\left( \xi ,u^{j_{n-1}}\left( \xi \right) \right) =\max_{u\in
U_{\xi }}\hat{\varkappa}\left( u\right) ,
\end{equation*}%
where $\hat{\varkappa}\left( u\right) :=\frac{1}{2\left\Vert \nabla f\left(
\xi \right) \right\Vert }\left\langle \nabla ^{2}f\left( \xi \right)
u,u\right\rangle $ and $U_{\xi }:=\left\{ v\in \mathbf{T}_{F}\left( \xi
\right) :\left\Vert v\right\Vert =1\right\} $.
\end{theorem}

\begin{proof}
Assuming (\ref{Norma}) it is easy to show that the application $u\mapsto 
\hat{\varkappa}\left( u\right) $ is continuous in $\mathbf{S}:=\left\{ x\in 
\mathbb{R}^{n}:\left\Vert x\right\Vert =1\right\} $, and in particular in $%
U_{\xi }.$ Hence it admits a maximum and a minimum on $U_{\xi }$. Let $%
\overline{u}\in U_{\xi }$ be a maximum point. Then, by (\ref{R3}),%
\begin{equation*}
\hat{\varkappa}\left( \overline{u}\right) =\max_{u\in U_{\xi }}\hat{\varkappa%
}\left( u\right) \geq \hat{\varkappa}\left( \frac{u^{j_{n-1}}\left( \xi
\right) }{\left\Vert u^{j_{n-1}}\left( \xi \right) \right\Vert }\right) =%
\hat{\varkappa}_{F}\left( \xi ,u^{j_{n-1}}\left( \xi \right) \right) .
\end{equation*}%
On the other hand, since $\overline{u}\in \mathbf{T}_{F}\left( \xi \right) ,$
then $\hat{\varkappa}_{F}\left( \xi ,\overline{u}\right) \leq \hat{\varkappa}%
_{F}\left( \xi ,u^{j_{n-1}}\left( \xi \right) \right) $, and consequently $%
\hat{\varkappa}\left( \overline{u}\right) =\hat{\varkappa}_{F}\left( \xi
,u^{j_{n-1}}\left( \xi \right) \right) .$ Reasoning as above, if $\overline{v%
}$ is a minimum on $U_{\xi }$, we deduce that $\hat{\varkappa}\left( 
\overline{v}\right) =\hat{\varkappa}_{F}\left( \xi ,u^{j_{1}}\left( \xi
\right) \right) .$
\end{proof}

\section{Directional curvature radius}

As in \cite[p.14]{Busemann} (see also \cite{Schneider, Adiprasito}) we also
relate the directional curvature to the radius of some ball.

\begin{definition}
The $\emph{2}$\emph{-dimensional curvature radius of }$F$\emph{\ at }$\xi
\in \partial F$ (w.r.t. $\xi ^{\ast }$) \emph{in the direction of }$%
u^{j}\left( \xi \right) ,$ $j\in I\backslash \left\{ i\right\} ,$ is given by%
\begin{equation}
\widehat{\mathfrak{R}}_{F}\left( \xi ,u^{j}\left( \xi \right) \right) =\frac{%
1}{2\hat{\varkappa}_{F}\left( \xi ,u^{j}\left( \xi \right) \right) }.
\label{raio}
\end{equation}
\end{definition}

Roughly speaking, the directional curvature $\hat{\varkappa}_{F}\left( \xi
,u^{j}\left( \xi \right) \right) $ shows how rotund the boundary $\partial F$
is in a neighbourhood of $\xi $ (watching from the end of the vector $\xi
^{\ast }$) when we "cut" $F$ with the plane $P\left( \xi ,u^{j}\left( \xi
\right) \right) $. As follows from Proposition \ref{00} it does not depend
on the position of the origin in $\mbox{int}F$ and can be defined also when $0\not\in \mbox{int}F$. By using (\ref{raio}) we give the following
geometric characterization of the directional curvature radius.

\begin{proposition}
Fixed $j\in I\backslash \left\{ i\right\} $, we have%
\begin{equation}
\frac{\widehat{\mathfrak{R}}_{F}\left( \xi ,u^{j}\left( \xi \right) \right) 
}{\left\Vert \xi ^{\ast }\right\Vert }=\underset{_{\substack{ \left(
\varepsilon ,\eta \right) \rightarrow \left( 0^{+},\xi \right)  \\ \eta \in
\partial F}}}{\lim \sup }\inf \left\{ r>0:F\cap P\left( \eta ,u^{j}\left(
\eta \right) \right) \cap \left( \eta +\varepsilon \overline{\mathbf{B}}%
\right) \subset \eta -r\eta ^{\ast }+r\left\Vert \eta ^{\ast }\right\Vert 
\overline{\mathbf{B}}\right\}.  \label{06}
\end{equation}
\end{proposition}

\begin{proof}
Let us prove first the inequality $"\leq"$ in (\ref{06}) assuming without
loss of generality that the right-hand side (further denoted by $R$) is
finite. Taking an arbitrary $\rho >R$, by the definition of $\lim \sup $, we
can afirm that for each $\varepsilon >0$ small enough and for each $\eta \in
\partial F$ from a neighbourhood of $\xi $, the relation%
\begin{equation*}
\inf \left\{ r>0:F\cap P\left( \eta ,u^{j}\left( \eta \right) \right) \cap
\left( \eta +\varepsilon \overline{\mathbf{B}}\right) \subset \eta -r\eta
^{\ast }+r\left\Vert \eta ^{\ast }\right\Vert \overline{\mathbf{B}}\right\}
<\rho
\end{equation*}%
holds. In particular,
\begin{equation*}
F\cap P\left( \eta ,u^{j}\left( \eta \right) \right) \cap \left( \eta
+\varepsilon \overline{\mathbf{B}}\right) \subset \eta -\rho \eta ^{\ast
}+\rho \left\Vert \eta ^{\ast }\right\Vert \overline{\mathbf{B}}\mathbf{,}
\end{equation*}%
which implies 
\begin{equation*}
\left\Vert \zeta -\eta +\rho \eta ^{\ast }\right\Vert ^{2}\leq \rho
^{2}\left\Vert \eta ^{\ast }\right\Vert ^{2},
\end{equation*}%
whenever $\zeta \in F\cap P\left( \eta ,u^{j}\left( \eta \right) \right) $
with $\left\Vert \zeta -\eta \right\Vert =\varepsilon $, or, in another form,
\begin{equation}
\left\langle \zeta -\eta ,\eta ^{\ast }\right\rangle \leq -\frac{\varepsilon
^{2}}{2\rho }.  \label{19}
\end{equation}%
If $w\in F\cap P\left( \eta ,u^{j}\left( \eta \right) \right) $ is an
arbitrary point with $\left\Vert w-\eta \right\Vert \geq \varepsilon $ then
setting $\zeta :=\lambda w+\left( 1-\lambda \right) \eta $ in $F\cap P\left(
\eta ,u^{j}\left( \eta \right) \right) $, where $\lambda :=\frac{\varepsilon 
}{\left\Vert w-\eta \right\Vert }\leq 1$, we have%
\begin{equation*}
\left\Vert \zeta -\eta \right\Vert =\left\Vert \lambda w+\left( 1-\lambda
\right) \eta -\eta \right\Vert =\lambda \left\Vert w-\eta \right\Vert
=\varepsilon
\end{equation*}%
and%
\begin{equation*}
\left\langle \eta -\zeta ,\eta ^{\ast }\right\rangle =\lambda \left\langle
\eta -w,\eta ^{\ast }\right\rangle .
\end{equation*}%
By (\ref{19}) we obtain 
\begin{equation*}
\frac{\varepsilon ^{2}}{2\rho }\leq \left\langle \eta -\zeta ,\eta ^{\ast
}\right\rangle =\lambda \left\langle \eta -w,\eta ^{\ast }\right\rangle \leq
\left\langle \eta -w,\eta ^{\ast }\right\rangle ,
\end{equation*}%
so%
\begin{equation*}
\frac{\varepsilon ^{2}}{2\rho }\leq \inf \left\{ \left\langle \eta -w,\eta
^{\ast }\right\rangle :w\in F\cap P\left( \eta ,u^{j}\left( \eta \right)
\right) ,\left\Vert w-\eta \right\Vert \geq \varepsilon \right\} .
\end{equation*}%
Hence, passing to $\lim \inf $ as $\varepsilon \rightarrow 0^{+},\eta
\rightarrow \xi $ and $\rho \rightarrow R^{+}$ we conclude the first part of
the proof.

In order to show the opposite inequality let us assume that $R>0$ (the case $%
R=0$ is trivial). If $0<\rho <R$ then, by the definition of $\lim \sup $
there are $\varepsilon >0$ arbitrarily small and $\eta \in \partial F$
arbitrarily closed to $\xi $, such that%
\begin{equation*}
\inf \left\{ r>0:F\cap P\left( \eta ,u^{j}\left( \eta \right) \right) \cap
\left( \eta +\varepsilon \overline{\mathbf{B}}\right) \subset \eta -r\eta
^{\ast }+r\left\Vert \eta ^{\ast }\right\Vert \overline{\mathbf{B}}\right\}
>\rho .
\end{equation*}%
Then the set $F\cap P\left( \eta ,u^{j}\left( \eta \right) \right) \cap
\left( \eta +\varepsilon \overline{\mathbf{B}}\right) $ is not contained in $%
\eta -\rho \eta ^{\ast }+\rho \left\Vert \eta ^{\ast }\right\Vert \overline{%
\mathbf{B}}$, or, in other words, there is $\zeta \in F\cap P\left( \eta
,u^{j}\left( \eta \right) \right) $ with $\left\Vert \zeta -\eta \right\Vert
\leq \varepsilon $ such that 
\begin{equation*}
\left\Vert \zeta -\eta +\rho \eta ^{\ast }\right\Vert >\rho \left\Vert \eta
^{\ast }\right\Vert .
\end{equation*}%
Consequently, setting, $r:=\left\Vert \zeta -\eta \right\Vert \leq
\varepsilon $ we have%
\begin{equation*}
2\rho \left\langle \eta -\zeta ,\eta ^{\ast }\right\rangle <\left\Vert \zeta
-\eta \right\Vert ^{2}=r^{2}.
\end{equation*}%
So
\begin{eqnarray*}
\widehat{\mathfrak{C}}_{F}\left( r,\eta ,u^{j}\left( \eta \right) \right)
&=&\inf \left\{ \left\langle \eta -w,\eta ^{\ast }\right\rangle :w\in F\cap
P\left( \eta ,u^{j}\left( \eta \right) \right) ,\left\Vert w-\eta
\right\Vert \geq r\right\} \\
&\leq &\left\langle \eta -\zeta ,\eta ^{\ast }\right\rangle <\frac{r^{2}}{%
2\rho }.
\end{eqnarray*}%
Passing to $\lim \inf $ as $r\rightarrow 0^{+},$ $\eta \rightarrow \xi $ and
then to $\lim $ as $\rho \rightarrow R^{-}$ we conclude the proof.
\end{proof}

\section{Relation with the usual curvature formula for implicit space curves}

In this section we will compare the formula obtained in Theorem \ref{Teo1}
with the usual curvature formula for implicit space curves (that is, curves
in $\mathbb{R}^{n}$ generated by the intersection of $n-1$ implicit
equations). More precisely we will compare (\ref{R3}) with the formula
obtained by R. Goldman in \cite{Goldman}. To do this we will consider,
separately, the cases $n=2$ and $n\geq 3$.

For $n=2$, under our assumptions, near a fixed $\xi \in \partial F$ the
curve $\partial F\cap P\left( \xi ,u\left( \xi \right) \right) $ is given by 
$f\left( \eta \right) =0,$ and the tangent line at $\xi $ is $\mathbf{T}%
_{F}\left( \xi \right) =\mbox{span}\left\{ \left( -f_{2}\left( \xi
\right) ,f_{1}\left( \xi \right) \right) \right\} $ (see before Proposition %
\ref{n2}). By (\ref{Cn2}), for any $u\left( \xi \right) \in \mathbf{T}%
_{F}\left( \xi \right) \backslash \left\{ \mathbf{0}\right\} ,$ that is for $%
u\left( \xi \right) =\lambda \left( -f_{2}\left( \xi \right) ,f_{1}\left(
\xi \right) \right) $, with $\lambda \in \mathbb{R}\backslash \left\{
0\right\} $, we have
\begin{eqnarray*}
\hat{\varkappa}_{F}\left( \xi ,u\left( \xi \right) \right) &=&\frac{1}{%
2\left\Vert \nabla f\left( \xi \right) \right\Vert \left\Vert u\left( \xi
\right) \right\Vert ^{2}}\left\langle \nabla ^{2}f\left( \xi \right) u\left(
\xi \right) ,u\left( \xi \right) \right\rangle \\
&=&\frac{\left[ 
\begin{array}{cc}
-f_{2}\left( \xi \right) & f_{1}\left( \xi \right)%
\end{array}%
\right] \left[ 
\begin{array}{cc}
f_{11}\left( \xi \right) & f_{12}\left( \xi \right) \\ 
f_{12}\left( \xi \right) & f_{22}\left( \xi \right)%
\end{array}%
\right] \left[ 
\begin{array}{c}
-f_{2}\left( \xi \right) \\ 
f_{1}\left( \xi \right)%
\end{array}%
\right] }{2\left( f_{1}^{2}\left( \xi \right) +f_{2}^{2}\left( \xi \right)
\right) ^{\frac{3}{2}}}=\frac{1}{2}k_{G}\left( \xi \right) ,
\end{eqnarray*}%
where $k_{G}\left( \xi \right) $ is the curvature given by R. Goldman in 
\cite[(3.4)]{Goldman}. So, when $n=2$ the formulas coincide to less than the
constant $\frac{1}{2}.$ As we will see, we will obtain the same conclusion
for $n\geq 3$, but, in this case, the formula in \cite{Goldman} is more
difficult to apply than ours.

Before analyzing the case $n\geq 3,$ we need to introduce a generalization
to the cross product from $3$-dimensions to $n$-dimensions, called external
product.

\begin{definition}
\cite[p.165]{Greub} The \emph{external product} of two vectors in an $n$%
-dimensional space, $n\geq 3$, spanned by $e_{1},...,e_{n}$ is a vector in a
space of dimension $\frac{n\left( n-1\right) }{2}$ spanned by a new
collection of vectors denoted by $\left\{ e_{i}\wedge e_{j}\right\} $, where 
$i<j$. Let $u=u_{1}e_{1}+...+u_{n}e_{n}$ and $v=v_{1}e_{1}+...+v_{n}e_{n}$
then%
\begin{equation*}
u\wedge v=\sum_{i<j}\det \left[ 
\begin{array}{cc}
u_{i} & u_{j} \\ 
v_{i} & v_{j}%
\end{array}%
\right] \left( e_{i}\wedge e_{j}\right) .
\end{equation*}
\end{definition}

For the next definition, as well as for the rest of the work, we just need
to compute the magnitude of the external product, which is given by the
formula%
\begin{equation}
\left\Vert u\wedge v\right\Vert ^{2}=\sum_{i<j}\left( \det \left[ 
\begin{array}{cc}
u_{i} & u_{j} \\ 
v_{i} & v_{j}%
\end{array}%
\right] \right) ^{2}.  \label{quadrado}
\end{equation}%
Assuming that $e_{i}$, $i=1,...,n,$ is the vector of $\ \mathbb{R}^{n} \ $ with
one in the $i$th position and zero everywhere else, and that $e:=\left(
e_{1},...,e_{n}\right) $ is the canonical basis of $\mathbb{R}^{n}$, we are
in conditions to see the usual curvature formula for implicit space curves
(see, for example, \cite[(5.4)]{Goldman}).

\begin{definition}
\cite[(5.4)]{Goldman}\label{G} The curvature formula for a point $\xi $ on
the curve defined by the intersection of $n-1$ implicit hypersurfaces $%
F_{1}\left( x_{1},...,x_{n}\right) =0,...,F_{n-1}\left(
x_{1},...,x_{n}\right) =0$ is 
\begin{equation}
k_{G}\left( \xi \right) =\frac{\left\Vert \left( \mbox{Tan}\left(
F_{1},...,F_{n_{-1}}\right) \left( \xi \right) \ast \nabla \left(\mbox{Tan}\left( F_{1},...,F_{n_{-1}}\right) \right) \left( \xi \right) \right)
\wedge \mbox{Tan}\left( F_{1},...,F_{n_{-1}}\right) \left( \xi \right)
\right\Vert }{\left\Vert \mbox{Tan}\left( F_{1},...,F_{n_{-1}}\right)\left( \xi \right) \right\Vert ^{3}},  \label{goldman}
\end{equation}%
where $\mbox{Tan}\left( F_{1},...,F_{n_{-1}}\right) (\xi )$ is the
tangent to the intersection curve at $\xi $ given by%
\begin{eqnarray*}
\mbox{Tan}\left( F_{1},...,F_{n_{-1}}\right) \left( \xi \right) &=&\det %
\left[ 
\begin{array}{c}
e \\ 
\nabla F_{1}\left( \xi \right) \\ 
\vdots \\ 
\nabla F_{n-1}\left( \xi \right)%
\end{array}%
\right] \\
&=&\det \left[ 
\begin{array}{ccc}
e_{1} & \cdots & e_{n} \\ 
F_{1x_{1}}\left( \xi \right) & \cdots & F_{1x_{n}}\left( \xi \right) \\ 
\vdots & \vdots & \vdots \\ 
F_{n-1x_{1}}\left( \xi \right) & \cdots & F_{n-1x_{n}}\left( \xi \right)%
\end{array}\right],
\end{eqnarray*}%
$\nabla \left( \mbox{Tan}\left( F_{1},...,F_{n_{-1}}\right) \right)
\left( \xi \right) $ is the $n\times n$ matrix in where each column is the
gradiente of the respective component of the row matrix $\mbox{Tan}\left(
F_{1},...,F_{n_{-1}}\right)$, with the derivatives calculated at $\xi $,
and $\ast $ represents the product between matrices.
\end{definition}

To make the desired comparison, we still need to introduce some notation.
Fixed $\xi \in \partial F$, $i\in I$ (given by (\ref{fi})) and $j\in
I\backslash \left\{ i\right\} $, for any $k\in I\backslash \left\{
i,j\right\} $ put%
\begin{equation*}
a_{kl}\left( \xi \right) :=-\frac{f_{x_{l}}\left( \xi \right)
f_{x_{k}}\left( \xi \right) }{f_{x_{i}}^{2}\left( \xi \right)
+f_{x_{j}}^{2}\left( \xi \right) },\text{ }\ \ l=i,j,
\end{equation*}%
and%
\begin{equation*}
p_{k\xi }\left( \eta _{1},...,\eta _{n}\right) :=\eta _{k}-\xi
_{k}+a_{ki}\left( \xi \right) \left( \eta _{i}-\xi _{i}\right) +a_{kj}\left(
\xi \right) \left( \eta _{j}-\xi _{j}\right) \text{, \ \ \ \ }\left( \eta
_{1},...,\eta _{n}\right) \in \mathbb{R}^{n}.
\end{equation*}%
Using the definition of generated space it is easy to show that%
\begin{equation*}
P\left( \xi ,u^{j}\left( \xi \right) \right) =\bigcap\limits_{k\in
I\backslash \left\{ i,j\right\} }\left\{ \eta \in \mathbb{R}^{n}:p_{k\xi
}\left( \eta \right) =0\right\} .
\end{equation*}%
At the neighbourhood $\xi +\delta ^{\prime }\mathbf{B}$ ($\delta ^{\prime
}>0 $ is given by Remark \ref{r1}) the curve $\partial F\cap P\left( \xi
,u^{j}\left( \xi \right) \right) $ is given by the intersection of $n-1$
implicit equations:%
\begin{equation*}
f\left( \eta \right) =0,\ \ \ \ p_{k_{1}\xi }\left( \eta \right)
=0,...,p_{k_{n-2}\xi }\left( \eta \right) =0,
\end{equation*}%
$k_{1},...,k_{n-2}\in I\backslash \left\{ i,j\right\} $ and $%
k_{1}<...<k_{n-2}.$

Next we will compute the curvature for this curve in the sense of Definition %
\ref{G}.

\begin{theorem}
\label{Teo2}We have%
\begin{equation*}
k_{G}\left( \xi \right) =\frac{\left\vert f_{x_{i}x_{i}}\left( \xi \right)
f_{x_{j}}^{2}\left( \xi \right) -2f_{x_{i}}\left( \xi \right)
f_{x_{j}}\left( \xi \right) f_{x_{i}x_{j}}\left( \xi \right)
+f_{x_{j}x_{j}}\left( \xi \right) f_{x_{i}}^{2}\left( \xi \right)
\right\vert }{\left\Vert \nabla f\left( \xi \right) \right\Vert \left(
f_{x_{i}}^{2}\left( \xi \right) +f_{x_{j}}^{2}\left( \xi \right) \right) },
\end{equation*}%
where%
\begin{equation*}
f_{x_{l}x_{m}}\left( \xi \right) :=\frac{\partial }{\partial x_{m}}\left( 
\frac{\partial f}{\partial x_{l}}\right) \left( \xi \right) ,\text{ }\ \ \
m,l\in \left\{ i,j\right\} .
\end{equation*}
\end{theorem}

\begin{proof}
For $k\in I\backslash \left\{ i,j\right\} $ and $x:=\left(
x_{1},...,x_{n}\right) \in \mathbb{R}^{n}$ fixed, we have%
\begin{equation}
\frac{\partial p_{k\xi }}{\partial x_{m}}\left( x\right) =\left\{ 
\begin{array}{ll}
1, & \text{if }m=k \\ 
a_{km}\left( \xi \right) , & \text{if }m=i\text{ or }m=j \\ 
0, & \text{otherwise}%
\end{array}%
\right. .  \label{deriv}
\end{equation}%
So the tangent to the intersection curve $\partial F\cap P\left( \xi
,u^{j}\left( \xi \right) \right) $ at $\eta \in \left( \xi +\delta ^{\prime }%
\mathbf{B}\right) \cap \partial F\cap P\left( \xi ,u^{j}\left( \xi \right)
\right) $ is given by 
\begin{eqnarray*}
\mbox{Tan}\left( f,p_{k_{1}\xi },...,p_{k_{n-2}\xi }\right) \left( \eta
\right) &=&\det \left[ 
\begin{array}{c}
e \\ 
\nabla f\left( \eta \right) \\ 
\nabla p_{k_{1}\xi }\left( \eta \right) \\ 
\vdots \\ 
\nabla p_{k_{n-2}\xi }\left( \eta \right)%
\end{array}%
\right] \\
&=&\sum_{m=1}^{n}\left( -1\right) ^{1+m}\det A_{m\xi }\left( \eta \right)
e_{m},
\end{eqnarray*}%
where, for each $m\in I$, $A_{m\xi }\left( \eta \right) $ is the matrix
obtained from $\left[ 
\begin{array}{c}
e \\ 
\nabla f\left( \eta \right) \\ 
\nabla p_{k_{1}\xi }\left( \eta \right) \\ 
\vdots \\ 
\nabla p_{k_{n-2}\xi }\left( \eta \right)%
\end{array}%
\right] $ eliminating the first line and the $m$th column. Remembering that
we have (\ref{deriv}) for each $r\in \left\{ 1,n-2\right\} $, then%
\begin{equation*}
\left( -1\right) ^{1+m}\det A_{m\xi }\left( \eta \right) =\left( -1\right)
^{i+j}\left\{ 
\begin{array}{cc}
\sum_{r=1}^{n-2}f_{x_{k_{r}}}\left( \eta \right) a_{k_{r}j}\left( \xi
\right) -f_{x_{j}}\left( \eta \right) ,\medskip & \text{if }i<j\text{ and }%
m=i \\ 
f_{x_{i}}\left( \eta \right) -\sum_{r=1}^{n-2}f_{x_{k_{r}}}\left( \eta
\right) a_{k_{r}i}\left( \xi \right) ,\medskip & \text{if }i<j\text{ and }m=j
\\ 
f_{x_{j}}\left( \eta \right) -\sum_{r=1}^{n-2}f_{x_{k_{r}}}\left( \eta
\right) a_{k_{r}j}\left( \xi \right) ,\medskip & \text{if }i>j\text{ and }m=i
\\ 
\sum_{r=1}^{n-2}f_{x_{k_{r}}}\left( \eta \right) a_{k_{r}i}\left( \xi
\right) -f_{x_{i}}\left( \eta \right) ,\medskip & \text{if }i>j\text{ and }%
m=j \\ 
f_{x_{j}}\left( \eta \right) a_{k_{m}i}\left( \xi \right) -f_{x_{i}}\left(
\eta \right) a_{k_{m}j}\left( \xi \right) ,\medskip & \text{otherwise.}%
\end{array}%
\right.
\end{equation*}%
Let $\omega ^{j}\left( \xi \right) \in \mathbb{R}^{n}$ the vector with $%
\left( -1\right) ^{i+j+1}f_{x_{j}}\left( \xi \right) $ in the $i$th
coordinate, $\left( -1\right) ^{i+j}f_{x_{i}}\left( \xi \right) $ at the $j$%
th coordinate, if $i<j$, or with symmetrical values if $i>j$, and $0$
elsewhere. It is easy to show that 
\begin{eqnarray*}
\sum_{m=1}^{n}\left( -1\right) ^{1+m}\det A_{m}\left( \xi \right) e_{m} &=&%
\frac{\left\Vert \nabla f\left( \xi \right) \right\Vert ^{2}}{%
f_{x_{i}}^{2}\left( \xi \right) +f_{x_{j}}^{2}\left( \xi \right) } \ \omega
^{j}\left( \xi \right) \\
&=&\left( -1\right) ^{i+j}\frac{\left\Vert \nabla f\left( \xi \right)
\right\Vert ^{2}f_{x_{i}}\left( \xi \right) }{f_{x_{i}}^{2}\left( \xi
\right) +f_{x_{j}}^{2}\left( \xi \right) } \ u^{j}\left( \xi \right),
\end{eqnarray*}%
that is 
\begin{equation}
\mbox{Tan}\left( f,p_{k_{1}},...,p_{k_{n-2}}\right) \left( \xi \right)
=\left( -1\right) ^{i+j}\frac{\left\Vert \nabla f\left( \xi \right)
\right\Vert ^{2}f_{x_{i}}\left( \xi \right) }{f_{x_{i}}^{2}\left( \xi
\right) +f_{x_{j}}^{2}\left( \xi \right) } \ u^{j}\left( \xi \right).
\label{tan}
\end{equation}%
After some calculations we conclude that%
\begin{equation*}
\mbox{Tan}\left( f,p_{k_{1}},...,p_{k_{n-2}}\right) \left( \xi \right)
\ast \nabla \left( \mbox{Tan}\left( f,p_{k_{1}},...,p_{k_{n-2}}\right)
\right) \left( \xi \right) =\frac{\left\Vert \nabla f\left( \xi \right)
\right\Vert ^{2}}{f_{x_{i}}^{2}\left( \xi \right) +f_{x_{j}}^{2}\left( \xi
\right) } \ M,
\end{equation*}%
where $M$ is a line matrix. Using (\ref{quadrado}) we obtain%
\begin{eqnarray*}
&&\left\Vert \left( \mbox{Tan}\left( f,p_{k_{1}},...,p_{k_{n-2}}\right)
\left( \xi \right) \ast \nabla \left( \mbox{Tan}\left(
f,p_{k_{1}},...,p_{k_{n-2}}\right) \right) \left( \xi \right) \right) \wedge 
\mbox{Tan}\left( f,p_{k_{1}},...,p_{k_{n-2}}\right) \left( \xi \right)
\right\Vert ^{2} \\
&=&\left( \frac{\left\Vert \nabla f\left( \xi \right) \right\Vert ^{2}}{%
f_{x_{i}}^{2}\left( \xi \right) +f_{x_{j}}^{2}\left( \xi \right) }\right)
^{5}\left(f_{x_{i}x_{i}}\left( \xi \right) f_{x_{j}}^{2}\left( \xi \right)
-2f_{x_{i}}\left( \xi \right) f_{x_{j}}\left( \xi \right)
f_{x_{i}x_{j}}\left( \xi \right) +f_{x_{j}x_{j}}\left( \xi \right)
f_{x_{i}}^{2}\left( \xi \right) \right) ^{2}.
\end{eqnarray*}%
Which implies%
\begin{eqnarray*}
&&\left\Vert \left( \mbox{Tan}\left( f,p_{k_{1}},...,p_{k_{n-2}}\right)
\left( \xi \right) \ast \nabla \left( \mbox{Tan}\left(
f,p_{k_{1}},...,p_{k_{n-2}}\right) \right) \left( \xi \right) \right) \wedge 
\mbox{Tan}\left( f,p_{k_{1}},...,p_{k_{n-2}}\right) \left( \xi \right)
\right\Vert \\
&=&\frac{\left\Vert \nabla f\left( \xi \right) \right\Vert ^{5}}{\left(
f_{x_{i}}^{2}\left( \xi \right) +f_{x_{j}}^{2}\left( \xi \right) \right) ^{%
\frac{5}{2}}}\left\vert f_{x_{i}x_{i}}\left( \xi \right) f_{x_{j}}^{2}\left(
\xi \right) -2f_{x_{i}}\left( \xi \right) f_{x_{j}}\left( \xi \right)
f_{x_{i}x_{j}}\left( \xi \right) +f_{x_{j}x_{j}}\left( \xi \right)
f_{x_{i}}^{2}\left( \xi \right) \right\vert ,
\end{eqnarray*}%
and consequently (see (\ref{tan}))%
\begin{eqnarray*}
k_{G}\left( \xi \right) &=&\frac{\frac{\left\Vert \nabla f\left( \xi \right)
\right\Vert ^{5}}{\left( f_{x_{i}}^{2}\left( \xi \right)
+f_{x_{j}}^{2}\left( \xi \right) \right) ^{\frac{5}{2}}}\left\vert
f_{x_{i}x_{i}}\left( \xi \right) f_{x_{j}}^{2}\left( \xi \right)
-2f_{x_{i}}\left( \xi \right) f_{x_{j}}\left( \xi \right)
f_{x_{i}x_{j}}\left( \xi \right) +f_{x_{j}x_{j}}\left( \xi \right)
f_{x_{i}}^{2}\left( \xi \right) \right\vert }{\left( \frac{\left\Vert \nabla
f\left( \xi \right) \right\Vert ^{2}\left\vert f_{x_{i}}\left( \xi \right)
\right\vert }{f_{x_{i}}^{2}\left( \xi \right) +f_{x_{j}}^{2}\left( \xi
\right) }\left\Vert u^{j}\left( \xi \right) \right\Vert \right) ^{3}} \\
&=&\frac{\left\vert f_{x_{i}x_{i}}\left( \xi \right) f_{x_{j}}^{2}\left( \xi
\right) -2f_{x_{i}}\left( \xi \right) f_{x_{j}}\left( \xi \right)
f_{x_{i}x_{j}}\left( \xi \right) +f_{x_{j}x_{j}}\left( \xi \right)
f_{x_{i}}^{2}\left( \xi \right) \right\vert }{\left\Vert \nabla f\left( \xi
\right) \right\Vert \left( f_{x_{i}}^{2}\left( \xi \right)
+f_{x_{j}}^{2}\left( \xi \right) \right) }.
\end{eqnarray*}
\end{proof}\\

Remembering (\ref{R3}), and that $\hat{\varkappa}_{F}\left( \xi ,u^{j}\left(
\xi \right) \right) \geq 0$ because $\widehat{\mathfrak{C}}_{F}\left( r,\eta
,u^{j}\left( \xi \right) \right) \geq 0$ for every $r>0$ and every $\eta \in
\partial F$ close enough to $\xi $, it is easy to show the next result.

\begin{corollary}
We have%
\begin{equation}
k_{G}\left( \xi \right) =2\hat{\varkappa}_{F}\left( \xi ,u^{j}\left( \xi
\right) \right) .  \label{igualdade}
\end{equation}
\end{corollary}

Furthermore, following the proof of Theorem \ref{Teo2} it is easy to see
that we have $k_{G}\left( \eta \right) =2\hat{\varkappa}_{F}\left( \eta
,u^{j}\left( \eta \right) \right) $, for any $\eta \in \partial F$ close
enough to $\xi .$

So, we have the equality (\ref{igualdade}) for every $n\in \mathbb{N}$ and
for every $\eta \in \partial F$ close enough to $\xi $. For $n\geq 3,$ if
our conditions are verified, as we can easily see, the formula (\ref{goldman}%
) is harder to apply than formula (\ref{R3}), that is, it is more laborious
to calculate $k_{G}\left( \eta \right) $ than $\hat{\varkappa}_{F}\left(
\eta ,u^{j}\left( \eta \right) \right) $. Furthermore, the difficulty
increases as $n$ increases, and with (\ref{R3}) it is easy to calculate the
curvature in different directions, while using (\ref{goldman}) this would be
very laborious, since we would have to repeat all the calculations made in
the proof of Theorem \ref{Teo2} for each new $u^{j}$.

\section{Examples}

\begin{enumerate}
\item \label{Ex1}Consider the compact convex set $F\subset \mathbb{R}^{2}$,
with $\left( 0,0\right) $ in its interior, 
\begin{equation*}
F=\left\{ \left( x_{1},x_{2}\right) \in \mathbb{R}^{2}:\left\vert
x_{2}\right\vert \leq 1-x_{1}^{4},\;-1\leq x_{1}\leq 1\right\}.
\end{equation*}%
Close to $\xi =\left( \xi _{1},\xi _{2}\right) \in \partial F$ with $\xi
_{2}>0$ (the case $\xi _{2}<0$ is analogous) we have $f\left(
x_{1},x_{2}\right) :=x_{2}-1+x_{1}^{4}$, 
\begin{equation*}
\nabla f\left( \xi \right) =\left( 4\xi _{1}^{3},1\right),\text{ }\ \nabla
^{2}f\left( \xi \right) =\left[ 
\begin{array}{cc}
12\xi _{1}^{2} & 0 \\ 
0 & 0%
\end{array}%
\right],
\end{equation*}%
and%
\begin{equation*}
\mathbf{T}_{F}\left( \xi \right) =\mbox{span}\left\{ \left( 1,-4\xi
_{1}^{3}\right) \right\} .
\end{equation*}%
Consequently, for any $u\left( \xi \right) \in \mathbf{T}_{F}\left( \xi
\right) \backslash \left\{ \left( 0,0\right) \right\} ,$%
\begin{equation*}
\hat{\varkappa}_{F}\left( \xi ,u\left( \xi \right) \right) =\frac{12\xi
_{1}^{2}}{2\sqrt{\left( 4\xi _{1}^{3}\right) ^{2}+1} \ \left( \left( 4\xi
_{1}^{3}\right) ^{2}+1\right) }=\frac{6\xi _{1}^{2}}{\left( 16\xi
_{1}^{6}+1\right) ^{\frac{3}{2}}}.
\end{equation*}%
Recalling Proposition \ref{n2}, we can see that Theorem \ref{Teo1} allows us
to obtain the following equality 
\begin{equation*}
\hat{\varkappa}_{F}\left( \xi \right) =\frac{6\xi _{1}^{2}}{\left( 16\xi
_{1}^{6}+1\right) ^{\frac{3}{2}}},
\end{equation*}%
whereas in \cite[Example 8.3]{GP1} we had obtained only the inequalities 
\begin{equation*}
\frac{6\xi _{1}^{2}}{\sqrt{1+16\xi _{1}^{6}}\ \Sigma ^{2}\left( \xi
_{1}\right) }\leq \hat{\varkappa}_{F}\left( \xi \right) \leq \frac{6\xi
_{1}^{2}}{\sqrt{1+16\xi _{1}^{6}}},
\end{equation*}%
where $\Sigma \left( \xi _{1}\right) :=\sqrt{1+\left( \underset{k=0}{\overset{3} {\sum }}\left\vert \xi _{1}\right\vert ^{k}\right)^{2}}$.

Note that at $\xi =\left( 0,\pm 1\right) $ we have $\hat{\varkappa}%
_{F}\left( \xi ,u\left( \xi \right) \right) =0$, as would be expected. Here
we can't calculate the curvature at $\xi =\left( \pm 1,0\right) $ because
there isn't a $\mathcal{C}^{2}$ function $f$ checking our conditions, but in 
\cite[Example 8.3]{GP1} there is an estimate for the curvature at such
points.

\item Let $F$ a sphere in $\mathbb{R}^{n}$%
\begin{equation*}
\left\{ x=\left( x_{1},...,x_{n}\right) :\sum_{t=1}^{n}x_{t}^{2}\leq
R^{2}\right\} .
\end{equation*}%
Consider $f\left( x\right) =\sum_{t=1}^{n}x_{t}^{2}-R^{2}$ for $x$ near a fixed $\xi \in \partial F.$ We have%
\begin{equation*}
\nabla f\left( \xi \right) =2\xi ,\text{ }\ \nabla ^{2}f\left( \xi \right) =2%
\mathbf{I}_{n},
\end{equation*}%
where $\mathbf{I}_{n}$ is the identity matrix of the order $n$. Fix the
first $i\in I:=\left\{ 1,...,n\right\} $ such that $f_{x_{i}}\left( \xi
\right) =2\xi _{i}\neq 0$, then $\mathbf{T}_{F}\left( \xi \right) $ is
spanned by $n-1$ vectors $u^{j}\left( \xi \right) \in \mathbb{R}^{n}$, $j\in
I\backslash \left\{ i\right\} ,$ with $1$ in the $j$th coordinate$,$ $-\frac{%
\xi _{j}}{\xi _{i}}$ in the $i$th coordinate and $0$ in the others. Therefore
\begin{equation*}
\hat{\varkappa}_{F}\left( \xi ,u^{j}\left( \xi \right) \right) =\frac{%
2\left( \left( \frac{\xi _{j}}{\xi _{i}}\right) ^{2}+1\right) ^{2}}{%
4\left\Vert \xi \right\Vert \left\Vert u^{j}\left( \xi \right) \right\Vert
^{2}}=\frac{1}{2R},
\end{equation*}%
which means that the curvature at any point of the boundary of the sphere,
in the direction of any vector of its hyperplane tangent, is equal to $\frac{%
1}{2R}.$

\item Consider a cylinder%
\begin{equation*}
F_{a,b}=\left\{ \left( x_{1},x_{2},x_{3}\right) \in \mathbb{R}%
^{3}:x_{1}^{2}+x_{3}^{2}\leq a^{2},\ \left\vert x_{2}\right\vert \leq
b\right\} \text{, \ \ }a,b\in \mathbb{R}^{+}.
\end{equation*}%
Near $\xi =\left( \xi _{1},\xi _{2},\xi _{3}\right) \in \partial F_{a,b}$,
with $\xi _{1}^{2}+\xi _{3}^{2}=a^{2},\ \left\vert \xi _{2}\right\vert <b,$
put $f\left( x_{1},x_{2},x_{3}\right) :=x_{1}^{2}+x_{3}^{2}-a^{2}$. We have%
\begin{equation*}
\nabla f\left( \xi \right) =2\left( \xi _{1},0,\xi _{3}\right) ,\text{ }\
\nabla ^{2}f\left( \xi \right) =\left[ 
\begin{array}{ccc}
2 & 0 & 0 \\ 
0 & 0 & 0 \\ 
0 & 0 & 2%
\end{array}%
\right],
\end{equation*}%
and, fixed the first $i\in \left\{ 1,3\right\} $ such that $f_{x_{i}}\left(
\xi \right) \neq 0$, we have%
\begin{equation*}
\mathbf{T}_{F_{a,b}}\left( \xi \right) =\left\{ \left(
v_{1},v_{2},v_{3}\right) :v_{i}=-\frac{\xi _{j}}{\xi _{i}}v_{j}, \ j\in
\left\{ 1,3\right\} \backslash \left\{ i\right\}, \ v_{2}\in \mathbb{R}%
\right\}.
\end{equation*}%
Then%
\begin{equation*}
\left\langle \nabla ^{2}f\left( \xi \right) u^{j}\left( \xi \right)
,u^{j}\left( \xi \right) \right\rangle =\left\{ 
\begin{array}{ll}
0,\smallskip & \text{if }j=2 \\ 
2\left( \frac{\xi _{j}}{\xi _{i}}\right) ^{2}+2,\ \  & \text{if }j\in
\left\{ 1,3\right\} \backslash \left\{ i\right\}%
\end{array}%
\right.
\end{equation*}%
and consequently%
\begin{equation*}
\hat{\varkappa}_{F_{a,b}}\left( \xi ,u^{j}\left( \xi \right) \right)
=\left\{ 
\begin{array}{ll}
0,\smallskip & \text{if }j=2 \\ 
\frac{1}{2a},\ \  & \text{if }j\in \left\{ 1,3\right\} \backslash \left\{
i\right\}%
\end{array}%
\right. .
\end{equation*}%
If we consider $u\left( \xi \right) =\alpha u^{2}\left( \xi \right) +\beta
u^{j}\left( \xi \right) $, for any $\alpha ,\beta \in \mathbb{R}$ we will
obtain%
\begin{equation*}
\hat{\varkappa}_{F_{a,b}}\left( \xi ,u\left( \xi \right) \right) =\frac{%
\beta ^{2}a}{2\left( a^{2}\beta ^{2}+\xi _{1}^{2}\alpha ^{2}\right) }\in %
\left] 0,\frac{1}{2a}\right[ .
\end{equation*}

Now fix $\xi =\left( \xi _{1},\xi _{2},\xi _{3}\right) \in \partial F_{a,b}$
with $\xi _{1}^{2}+\xi _{3}^{2}<a^{2}$ and $\xi _{2}=b$ (the case $\xi
_{2}=-b$ is analogous). Near $\xi $ we have $f\left(
x_{1},x_{2},x_{3}\right) :=x_{2}-b,$ 
\begin{equation*}
\nabla f\left( \xi \right) =\left( 0,1,0\right) ,\text{ }\ \mathbf{T}%
_{F_{a,b}}\left( \xi \right) =\mbox{span}\left\{ \left( 1,0,0\right)
,\left( 0,0,1\right) \right\}
\end{equation*}%
and $\nabla ^{2}f\left( \xi \right) $ is the zero matrix of the order $n$. So%
\begin{equation*}
\hat{\varkappa}_{F_{a,b}}\left( \xi ,u\left( \xi \right) \right) =0,\text{ \
\ }\forall u\left( \xi \right) \in \mathbf{T}_{F_{a,b}}\left( \xi \right)
\backslash \left\{ \left( 0,0,0\right) \right\} .
\end{equation*}
\end{enumerate}

\bigskip

\paragraph{Acknowledgements}

I would like to thank Professor Giovanni Colombo, University of Padova
(Italy), who encouraged me to complete this work.

\bigskip

This paper has been partially supported by Center for Research in
Mathematics and Applications (CIMA) related with the Differential Equations
and Optimization (DEO) group, through the grant UIDB/04674/2020 of FCT-Funda%
\c{c}\~{a}o para a Ci\^{e}ncia e a Tecnologia, Portugal.\\\

\end{document}